\documentclass[11pt]{amsart}
\usepackage[dvips]{graphicx}
\usepackage{amscd}
\usepackage{amsmath}
\usepackage{amsxtra}
\usepackage{amsfonts}
\usepackage{amssymb}

\newtheorem{theorem}{Theorem}[section]
\newtheorem{corollary}[theorem]{Corollary}

\newtheorem{proposition}[theorem]{Proposition}

\theoremstyle{definition}
\newtheorem{definition}[theorem]{Definition}
\newtheorem{remark}[theorem]{Remark}

\newtheorem{example}[theorem]{Example}
\theoremstyle{remark}

\renewcommand{\theclaim}{\textup{\theclaim}}

\numberwithin{equation}{section}

\def\openone%{\hbox{\upshape \small1\kern-3.3pt\normalsize1}}

{\mathchoice

{\hbox{\upshape \small1\kern-3.3pt\normalsize1}}

{\hbox{\upshape \small1\kern-3.3pt\normalsize1}}

{\hbox{\upshape \tiny1\kern-2.3pt\SMALL1}}

{\hbox{\upshape \Tiny1\kern-2pt\tiny1}}}

\makeatletter

\newbox\ipbox

\newcommand{\ip}[2]{\left\langle #1\, , \,#2\right\rangle}
\newcommand{\diracb}[1]{\left\langle #1\mathrel{\mathchoice

{\setbox\ipbox=\hbox{$\displaystyle \left\langle\mathstrut
#1\right.$}

\vrule height\ht\ipbox width0.25pt depth\dp\ipbox}

{\setbox\ipbox=\hbox{$\textstyle \left\langle\mathstrut
#1\right.$}

\vrule height\ht\ipbox width0.25pt depth\dp\ipbox}

{\setbox\ipbox=\hbox{$\scriptstyle \left\langle\mathstrut
#1\right.$}

\vrule height\ht\ipbox width0.25pt depth\dp\ipbox}

{\setbox\ipbox=\hbox{$\scriptscriptstyle \left\langle\mathstrut
#1\right.$}

\vrule height\ht\ipbox width0.25pt depth\dp\ipbox}

}\right. }

\newcommand{\dirack}[1]{\left. \mathrel{\mathchoice

{\setbox\ipbox=\hbox{$\displaystyle \left.\mathstrut
#1\right\rangle$}

\vrule height\ht\ipbox width0.25pt depth\dp\ipbox}

{\setbox\ipbox=\hbox{$\textstyle \left.\mathstrut
#1\right\rangle$}

\vrule height\ht\ipbox width0.25pt depth\dp\ipbox}

{\setbox\ipbox=\hbox{$\scriptstyle \left.\mathstrut
#1\right\rangle$}

\vrule height\ht\ipbox width0.25pt depth\dp\ipbox}

{\setbox\ipbox=\hbox{$\scriptscriptstyle \left.\mathstrut
#1\right\rangle$}

\vrule height\ht\ipbox width0.25pt depth\dp\ipbox}

} #1\right\rangle}

\newcommand{\beq}{\begin{equation}}

\newcommand{\eeq}{\end{equation}}

\newcommand{\cj}[1]{\overline{#1}}

\newcommand{\bz}{\mathbb{Z}}

\newcommand{\B}{\mathcal{B}}
\newcommand{\br}{\mathbb{R}}
\newcommand{\bc}{\mathbb{C}}
\newcommand{\bt}{\mathbb{T}}

\def\blfootnote{\xdef\@thefnmark{}\@footnotetext}

\newcommand{\gm}{\gamma}

\input xy
\xyoption{all}
\usepackage{amssymb}

\newcommand{\abs}[1]{\lvert#1\rvert}

\newcommand{\norm}[1]{\lvert \lvert#1\rvert \lvert }

\newcommand{\Span}{\overline{\operatorname*{span}}}

\def\R{\mathbb{R}}
\def\N{\mathbb{N}}

\def\H{\mathcal{H}}

\def\-{^{-1}}
\def\B{\mathcal{B}}

\def\C{\mathbb{C}}
\def\Z{\mathbb{Z}}

\begin{document}

\title[Orthonormal bases generated by Cuntz algebras]{Orthonormal bases generated by Cuntz algebras}
\author{Dorin Ervin Dutkay}

\address{[Dorin Ervin Dutkay] University of Central Florida\\
	Department of Mathematics\\
	4000 Central Florida Blvd.\\
	P.O. Box 161364\\
	Orlando, FL 32816-1364\\
U.S.A.\\} \email{Dorin.Dutkay@ucf.edu}
\author{Gabriel Picioroaga}
\address{[Gabriel Picioroaga] University of South Dakota\\
          Department of Mathematical Sciences\\
          414 E. Clark St. \\
          Vermillion, SD 57069\\
U.S.A. \\} \email{Gabriel.Picioroaga@usd.edu}
\author{Myung-Sin Song}
\address{[Myung-Sin Song] Department of Mathematics and Statistics\\
	Southern Illinois University Edwardsville\\
	Campus Box 1653, Science Building\\
	Edwardsville, IL 62026\\
U.S.A. \\} \email{msong@siue.edu}

\thanks{} 

\begin{abstract}
 We show how some orthonormal bases can be generated by representations of the Cuntz algebra; these include Fourier bases on fractal measures, generalized Walsh bases on the unit interval and piecewise exponential bases on the middle third Cantor set. 
\end{abstract}
\maketitle \tableofcontents

\section{Introduction}
The Cuntz algebra $\mathcal O_N$, \cite{Cun77} is the $C^*$-algebra generated by $N$ isometries $S_i$, $i=0,\dots,N-1$ with the properties:
\begin{equation}
S_i^*S_j=\delta_{ij},\, i,j=0,\dots,N-1,\quad \sum_{i=0}^{N-1}S_iS_i^*=I.
\label{eqcu1}
\end{equation}
The Cuntz algebras are ubiquitous in analysis, but we draw our inspiration from wavelet theory. The role played by the Cuntz algebras in wavelet theory was described in the work of Bratteli and Jorgensen \cite{BrJo02a,BrJo02b,BrJo00,BrJo97}. Orthonormal wavelet bases are constructed from various choices of quadrature mirror filters (QMF) (see \cite{Dau92}). These filters are in one-to-one correspondence with certain representations of the Cuntz algebra. In section 2, we will show how the ideas of Bratteli and Jorgensen carry over without too much difficulty in a more general setting associated to some non-linear dynamics. We describe here this setting and give some examples.

\begin{definition}\label{def0.1}
Let $X$ be a compact metric space and $\mu$ a Borel probability measure on $X$. Let $r: X\rightarrow X$ an $N$-to-$1$ onto Borel measurable map, i.e.  $|r^{-1}(z)|=N$ for $\mu$.a.e. $z\in X$, where $|\cdot|$ indicates cardinality. We say that $\mu$ is strongly invariant (for $r$) if for every continuous function $f$ on $X$ the following invariance equation is satisfied:

\beq\label{e1}
\int f d\mu=\frac{1}{N}\int \sum_{r(w)=z}f(w)d\mu (z)
\eeq

\end{definition}
\medskip

{\bf Assumption.} In this paper $\mu$ will be a strongly invariant measure for the $N$-to-1 map $r:X\rightarrow X$ as in Definition \ref{def0.1}

\begin{example}\label{ex1}
Let $\bt=\{z\in\bc\mbox{ : }|z|=1\}$ be the unit circle. Let $r(z)=z^N$, $z\in \bt$. Let $\mu$ be the Haar measure on $\bt$. Then $\mu$ is strongly invariant. An equivalent system can be realized on$[0,1]$ with $r(x)=Nx\mbox{ mod }1$, $x\in [0,1]$ with the Lebesgue measure $dx$ on $[0,1]$. We can identify the unit circle $\bt$ with the unit interval $[0,1]$ by $z=e^{2\pi ix}$. 
\end{example}

% intermediate adition:

\begin{example}\label{exx1}
Let $\Gamma$ be a countable discrete abelian group. Let $\alpha:\Gamma\rightarrow\Gamma$ be an endomorphism of $\Gamma$ such that $\alpha(\Gamma)$ has finite index $N$ in $\Gamma$ and 
\beq\label{x1}
\quad \cap_{n\geq 0}\alpha^n(\Gamma)=\{0\}
\eeq
Let $\hat{\Gamma}$ be the compact dual group and let $\mu$ be the Haar measure on $\hat{\Gamma}$, $\mu(\hat{\Gamma})=1$. Denote by $\alpha^*$ the dual endomorphism on $\hat{\Gamma}$, $w\mapsto w\circ\alpha$ 
$(w\in\hat{\Gamma})$. Observe that $\alpha^*$ is surjective, $|\mbox{ Ker }\alpha^*|=N$ so $|{\alpha^*}^{-1}(z)|=N$ for all $z\in\hat{\Gamma}$, and condition (\ref{x1}) implies that 
$\cup_{n\geq 0}\mbox{Ker }\alpha^{*n}$ is dense in $\hat{\Gamma}$.

\end{example}

\begin{proposition}
The Haar measure on $\hat{\Gamma}$ is strongly invariant for $\alpha^*$.
\end{proposition}

\begin{proof}
To prove the strong invariance relation (\ref{e1}) it is enough to check it on characters on $\hat{\Gamma}$, which by Pontryagin duality are given by the elements of $\Gamma$ and we denote them 
by $e_{\gm}(w)=w(\gm)$, $\gm\in\Gamma$, $w\in\hat{\Gamma}$. Fix $\gm\in\Gamma$, $\gm\neq 0$. Pick an element $g_0\in\hat{\Gamma}$ such that $e_{\gm}(g_0)\neq 1$. We have

$$\int_{\hat{\Gamma}}\frac{1}{N}\sum_{\alpha^*(w)=z}e_{\gm}(w) d\mu(z)=  \int_{\hat{\Gamma}}\frac{1}{N}\sum_{\alpha^*(w)=z-\alpha^*(g_0)}e_{\gm}(w) d\mu(z)= \int_{\hat{\Gamma}}\frac{1}{N}\sum_{\alpha^*(u)=z}e_{\gm}(u-g_0) d\mu(z)$$

$$= \cj{e_{\gm}(g_0)} \int_{\hat{\Gamma}}\frac{1}{N}\sum_{\alpha^*(u)=z}e_{\gm}(u) d\mu(z).$$

Since $e_{\gm}(g_0)\neq 1$ it follows that $\int_{\hat{\Gamma}}\frac{1}{N}\sum_{\alpha^*(w)=z}e_{\gm}(w) d\mu(z)=0$. Since $\int_{\hat{\Gamma}}e_{\gm}(z) d\mu(z)=0$ the strong invariance of $\mu$ is obtained. 

\end{proof}

%end addition

\begin{example}\label{ex2}
We consider affine iterated function systems with no overlap. Let $R$ be a $d\times d$ expansive real matrix, i.e., all the eigenvalues of $R$ have absolute value strictly greater than 1.Let $B\subset \br^d$ a finite set such that $N=|B|$. Define the affine iterated function system 

\beq\label{e2}
\tau_b(x)=R^{-1}(x+b)\quad(x\in\br^d\mbox{, }b\in B)
\eeq
By \cite{Hut81} there exists a unique compact subset $X_B$ of $\br^d$ which satisfies the invariance equation

\beq\label{e3}
X_B=\cup_{b\in B}\tau_b(X_B)
\eeq

$X_B$ is called the attractor of the iterated function system $(\tau_b)_{b\in B}$. Moreover $X_B$ is given by 

\beq\label{e4}
X_B=\left\{ \sum_{k=1}^{\infty}R^{-k}b_k \mbox{ : }b_k\in B \mbox{ for all }k\geq 1\right\}
\eeq
Also, from \cite{Hut81}, there is a unique probability measure $\mu_B$ on $\br^d$ satisfying the invariance equation

\beq\label{e5}
\int f d\mu_B=\frac{1}{N}\sum_{b\in B}\int f\circ\tau_b d\mu_B
\eeq
for all continuous compactly supported functions $f$ on $\br$. 
We call $\mu_B$ the invariant measure for the IFS $(\tau_b)_{b\in B}$. 
By \cite{Hut81}, $\mu_B$ is supported on the attractor $X_B$.
We say that the IFS has no overlap if 
$\mu_B(\tau_b(X_B)\cap\tau_b'(X_B))=\emptyset\mbox{ for all }b\neq b'\mbox{ in }B$.
\par Assume that the IFS $(\tau_b)_{b\in B}$ has no overlap. Define the map $r:X_B\rightarrow X_B$ 

\beq\label{e6}
r(x)=\tau_b^{-1}(x), \mbox{   if }x\in \tau_b(X_B)
\eeq

\noindent Then $r$ is an $N$-to-$1$ onto map and $\mu_B$ is strongly invariant for $r$.
 Note that $r^{-1}(x)=\{\tau_b(x)\mbox{ : }b\in B \}$ for $\mu_B$.a.e. $x\in X_B$.
\end{example} 

\begin{example}\label{ex3}
Let $r$ be a rational map on the complex sphere $\bc_{\infty}$. Let $J$ be its Julia set. Then by \cite{Bro65}, \cite{ObPi72} there exists a stongly invariant measure $\mu$ supported on $J$, which is non-atomic. The Julia set is invariant for $r$ and the restriction $r:J\rightarrow J$ is a $N$-to-$1$ onto map where $N=deg(r)$. 
\end{example}

We will show in Section 2 Proposition \ref{pr11} how representations of the Cuntz algebra are obtained from a choice of a quadrature mirror filter (QMF) basis (Definition \ref{def1.8}. Then we show how QMF bases can be constructed using some unitary matrix valued functions (Theorem \ref{th2.9}). This gives us a large variety of representations of the Cuntz algebras, which we use in Section 3 to construct various orthonormal bases. 

The central result of the paper is Theorem \ref{diez}, where we present a general criterion for a Cuntz algebra representation to generate an orthonormal basis. As a corollary (Theorem \ref{th3.4}), when applied to some affine iterated function systems, we obtain a construction of piecewise exponential bases on some Cantor fractal measures which extends a result of Dutkay and Jorgensen \cite{DJ06}. In particular, we obtain piecewise exponential orthonormal bases on the middle third Cantor set (Example \ref{ex3.7}) which is known \cite{JoPe98} not to have any orthonormal bases of exponential functions. 

Another corollary to our Theorem \ref{diez} gives us a construction of generalized Walsh bases on the unit interval starting from any unitary $N\times N$ matrix with constant first row.

\section{QMF bases and representations of the Cuntz algebra}

\begin{definition}\label{def1.5}
A {\it quadrature mirror filter (QMF)} for $r$ is a function $m_0$ in $L^\infty(X,\mu)$ with the property that 
\begin{equation}
\frac{1}{N}\sum_{r(w)=z}|m_0(w)|^2=1,\quad(z\in X)
\label{eq1.5.1}
\end{equation}
\end{definition}

As shown by Dutkay and Jorgensen \cite{DuJo05,DuJo07}, every QMF gives rise to a wavelet theory. Various extra conditions on the filter $m_0$ will produce wavelets in $L^2(\br)$ \cite{Dau92}, on Cantor sets \cite{DuJo06w,MaPa11}, on Sierpinski gaskets \cite{Dan08} and many others.
\begin{theorem}\cite{DuJo05,DuJo07}\label{th1.6}
Let $m_0$ be a QMF for $r$. Then there exists a Hilbert space $\H$, a representation $\pi$ of $L^\infty(X)$ on $\H$, a unitary operator $U$ on $\H$ and a vector $\varphi$ in $\H$ such that 
\begin{enumerate}
	\item {\bf (Covariance)} 
	\begin{equation}
U\pi(f)U^*=\pi(f\circ r),\quad(f\in L^\infty(X))
\label{eqw1}
\end{equation}
\item {\bf (Scaling equation)} 
\begin{equation}
U\varphi=\pi(m_0)\varphi
\label{eqw2}
\end{equation}
\item {\bf (Orthogonality)}
\begin{equation}
\ip{\pi(f)\varphi}{\varphi}=\int f\,d\mu,\quad(f\in L^\infty(X))
\label{eqw3}
\end{equation}
\item {\bf (Density)}
\begin{equation}
\Span\left\{U^{-n}\pi(f)\varphi : f\in L^\infty(X),n\geq 0\right\}=\H
\label{eqw4}
\end{equation}
\end{enumerate}

\end{theorem}

\begin{definition}\label{def1.7}
The system $(\H,U,\pi,\varphi)$ in Theorem \ref{th1.6} is called {\it the wavelet representation associated to the QMF $m_0$}.
\end{definition}

To construct a multiresolution, as in \cite{Dau92}, for a wavelet representation, one needs a QMF basis. 

\begin{definition}\label{def1.8}
A {\it{QMF basis}} is a set of $N$ QMF's\quad$m_0, m_1,\dots,m_{N-1}$ such that
\beq\label{e12}
\frac{1}{N}\sum_{r(w)=z}m_i(w)\cj{m_j}(w)=\delta_{ij} , \quad(i,j\in \{0,\dots,N-1\}, z\in X)
\eeq
\end{definition}

We can interpret these conditions in terms of a conditional expectation:

\begin{definition}\label{def1.9} 
Let $\mathcal{B}$ be the Borel sigma-algebra on $X$ and $r^{-1}(\B)$ be the sigma-algebra \\
$r^{-1}(\B)=\{  r^{-1}(B)\mbox{ : }B\in \B \}$. Note that the $r^{-1}(\B)$-measurable functions are of the form $f\circ r$, where $f$ is Borel measurable.
\end{definition}

The conditional expectation from to $\B$ to $r^{-1}(\B)$ is defined by
\beq\label{e13}
\mathbb{E} (f)(z)=\frac{1}{N}\sum_{r(w)=z}f(w),\quad(z\in X)
\eeq
Alternatively $\mathbb{E}(f)$ can be defined, up to $\mu$-measure zero as a $r^{-1}(\B)$-measurable function such that 
\beq\label{e14}
\int fg\circ r d\mu=\int \mathbb{E}(f) g\circ r d\mu,\quad \mbox{for all } g\in L^{\infty}(X,\mu).
\eeq

\begin{proposition}\label{pr10}
A set of functions $(m_i)_{i=0}^{N-1}$ in $L^{\infty}(X,\mu)$ is a QMF basis if and only if 
\beq\label{e15}
\mathbb{E} (m_i\cj{m_j})=\delta_{ij},\quad(i,j\in \{0,\dots N-1\})
\eeq
In this case any function $f\in L^2(X,\mu)$ can be written in the QMF basis as 
\beq\label{e16}
f=\sum_{i=0}^{N-1}\mathbb{E} (f\cj{m_i})m_i
\eeq
\end{proposition}
\begin{proof}
The first statement is clear. For the second, define for $f\in L^2(X,\mu)$ the vector-valued function
$F(f)(z)=(f(w))_{r(w)=r(z)}\in \bc^N$. 
Note that the QMF basis property implies that \\$(F(\frac{1}{\sqrt{N}}m_i )(z))_{i=0}^{N-1}$ is an orthonormal basis in $\bc^n$. Then for $z\in X$ 
$$F(f)(z)=  \sum_{i=0}^{N-1}\ip{F(f)(z)}{F( \frac{1}{\sqrt{N}}m_i )(z)}_{\bc^N}F( \frac{1}{\sqrt{N}}m_i )(z) =\sum_{i=0}^{N-1}\mathbb{E}(f\cj{m_i})(z)F(m_i)(z) $$
Then looking at the first component (since $r(z)=r(z)$ one can take $w=z$) we get (\ref{e16}).   
\end{proof}

%your equation numbering jumped from 16 to 21
Next, we show how a QMF basis induces a representation of the Cuntz algebra. 
\begin{proposition}\label{pr11}
Let $(m_i)_{i=0}^{N-1}$ be a QMF basis. Define the operators on $L^2(X,\mu)$
\beq\label{e17}
S_i(f)=m_if\circ r,\quad i=0,\dots,N-1
\eeq
Then the operators $S_i$ are isometries and they form a representation of the Cuntz algebra $\mathcal{O}_N$, i.e.
\beq\label{e18}
 S_i^*S_j=\delta_{ij},\quad i,j=0,\dots,N-1,\quad\quad \sum_{i=0}^{N-1}S_iS_i^*=I
\eeq
The adjoint of $S_i$ is given by the formula
\beq\label{e19}
S_i^*(f)(z)=\frac{1}{N}\sum_{r(w)=z}\cj{m_i}(w)f(w)
\eeq
\end{proposition}

\begin{proof}
We compute the adjoint: take $f$, $g$ in $L^2(X,\mu)$. We use the strong invariance of $\mu$.
$$\ip{S_i^*f}{g}=\int f\cj{m}_i\cj{g\circ r}\,d\mu=\int\frac{1}{N}\sum_{r(w)=z}\cj{m_i}(w)f(w)\cj{g}(z)d\mu(z)$$
Then (\ref{e19}) follows. The Cuntz relations in (\ref{e18}) are then easily checked with Proposition \ref{pr10}.
\end{proof}

Every QMF basis generates a multiresolution for the wavelet representation associated to $m_0$. Since the ideas are simple and are the same as in the classical wavelet theory presented  in \cite{Dau92}, we omit the proof. Note though, that the intersection of the resolution spaces might be non-trivial (for example, if $m_0=1$ then $1$ is contained in this intersection). 

\begin{proposition}
Let $(m_i)_{i=0}^{N-1}$ be a QMF basis. Let $(\H, U, \pi, \varphi)$ be the wavelet representation associated to $m_0$. Define

%numbering coming back to 16-19. do you this proposition befeore Prop.11?

\beq\label{eqq1}
 V_0:=\Span\left\{\pi(f)\varphi : f\in L^\infty(X)\right\},\quad\quad V_n=U^{-n}V_0,\quad n\in\Z
 \eeq

\beq\label{eqq2}
\psi_i=U^{-1}\pi(m_i)\varphi,\quad i=1,\dots,N-1
\eeq

\beq\label{eqq3}
W_i:=\Span\left\{\pi(f)\psi_i : f\in L^\infty(X)\right\}
\eeq

Then

\begin{enumerate}

\item
$\cj{\cup_{n\in\Z} V_n}=\H$

\item
$V_1=V_0\oplus W_1\oplus \dots \oplus W_{N-1}$

\item
If $\cap_{n\in\Z}V_n=\{0\}$ then $$\bigoplus_{n\in\Z}U^n\left(  W_1\oplus \dots \oplus W_{N-1} \right)=\H$$

\end{enumerate}

\end{proposition}

A particular case which we will use in Section 3, is that of QMF bases generated by Hadamard matrices which are defined from a finite set $B$ and its spectrum $\Lambda$. 

\begin{definition} Denote by $e_\lambda(x):=e^{2\pi i \lambda\cdot x}$ for $\lambda, x\in\br^d$. 
Let $B$ be a finite subset of $\R^d$, $|B|=:N$. We say that a finite set $\Lambda$ in $\R^d$ is a spectrum for $B$ if $|\Lambda|=N$ and the matrix
$$\frac{1}{\sqrt{N}}[e^{2\pi ib\cdot\lambda}]_{b\in B}^{\lambda\in\Lambda}$$
is unitary. Let $B$ and $L$ be finite subsets of $\Z^d$, $|B|=:N=|L|$ and let $R$ be an expansive $d\times d$ integer matrix. We say that $(B,L)$ 
is a Hadamard pair with scaling factor $R$ if $L$ is a spectrum for $R^{-1}B$; equivalently, the matrix

$$\frac{1}{\sqrt{N}}[e^{2\pi iR^{-1}b\cdot l}]_{b\in B}^{l\in L}$$
is unitary.

\end{definition}

\begin{example} \label{ex2.10}
Consider the setting in Example \ref{ex2}. We have the following equivalence:

\begin{proposition}\label{pr2.8}
A finite set $\Lambda$ in $\br^d$ is a spectrum for $R^{-1} B$ if and only if $(e_{\lambda})_{\lambda\in\Lambda}$ is a QMF basis. Let $L$ be a finite subset of $\Z^d$. Then $(B,L)$  is a Hadamard pair with scaling factor $R$ if and only if $(e_l)_{l\in L}$ is a QMF basis. 

\end{proposition}

\begin{proof}
We have 

$$ \frac{1}{N}\sum_{r(w)=z} e_{\lambda}(w) \cj{ e_{\lambda'}(w)}=\frac{1}{N}\sum_{b\in B} e^{2\pi i\tau_b(z)\cdot (\lambda-\lambda')}=\frac{1}{N}\sum_{b\in B}e^{2\pi i R^{-1}(z+b)\cdot (\lambda-\lambda')}$$
$$=e^{2\pi iR^{-1}(z)\cdot(\lambda-\lambda')}\mbox{ } \frac{1}{N}\sum_{b\in B}e^{2\pi iR^{-1}b\cdot (\lambda-\lambda')} $$
Thus, the QMF basis condition is equivalent to 
$$ \frac{1}{N}\sum_{b\in B}e^{2\pi iR^{-1}b\cdot (\lambda-\lambda')}=\delta_{\lambda\lambda'} $$
which is exactly the orthogonality of the columns of the matrix
$$\frac{1}{\sqrt{N}}[e^{2\pi iR^{-1}b\cdot\lambda}]_{b\in B}^{\lambda\in\Lambda}$$
The equivalence for Hadamard pairs follows as a particular case.

\end{proof}

If $B$ is a finite set and $R^{-1}B$ has spectrum $\Lambda$, then the set $\{e_\lambda:\lambda\in\Lambda\}$ is a QMF basis, by Proposition \ref{pr2.8}. Then, with Proposition \ref{pr11}, the operators $S_\lambda f=e_\lambda f\circ r$ form a representation of the Cuntz algebra. Such representations were studied in \cite{DuJo12cu}. 
\end{example}

The next theorem shows how QMF bases can be constructed from unitary matrix valued functions as in the work of Bratteli and Jorgensen \cite{BrJo02a,BrJo02b,BrJo00,BrJo97}, now in a more general context. 

\begin{theorem}\label{th2.9}
Fix $(m_i)_{i=0}^{N-1}$ a QMF basis. There is a one-to-one correspondence between the following two sets:

\begin{enumerate}
\item
QMF bases $(m_i')_{i=0}^{N-1}$

\item
Unitary valued maps $A:X\rightarrow U_N(\bc)$
\end{enumerate}
Given a QMF basis  $(m_i')_{i=0}^{N-1}$ the matrix $A$ with entries \\
$$(1)\quad A_{ij}(z)=\frac{1}{N}\sum_{r(w)=z} m_i'(w)\cj{m}_j(w),\quad (z\in X, i, j=0,\dots,N-1)$$
is unitary.\\
Given a unitary-valued map $A:X\rightarrow U_N(\bc)$, the functions form a QMF basis
$$(2)\quad m_i'(z)=\sum_{j=0}^{N-1}A_{ij}(r(z))m_j(z),\quad (z\in X, i=0,\dots N-1)$$
 These correspondences are inverse to each other.
\end{theorem}

\begin{proof}
The result requires some simple computations
$$\sum_{j=0}^{N-1}A_{ij}(z)\cj{A_{i'j}(z)}=\frac{1}{N^2}\sum_j\sum_{r(w)=z}m'_i(w)\cj{m_j(z)}\cdot \sum_{r(w')=z}\cj{m'_{i'}(w')\cj{m_j(w')}} =$$
$$\frac{1}{N^2}\sum_{w,w'}m'_i(w)\cj{m'_{i'}(w')}\cdot \sum_{j}\cj{m_j(w)}m_j(w')=\frac{1}{N}\sum_{w,w'}m_i'(w)\cj{m'_{i'}(w')}\delta_{w,w'}=\delta_{ii'}$$
Note that we used the equality 
$$\sum_{j}\cj{m_j(w)}m_j(w')=\delta_{ww'}$$
which follows from the fact that the matrix 
$$\frac{1}{\sqrt{N}}\left[  m_i(w) \right]_{w\in r^{-1}(z)}^{i=0,\dots N-1}$$
is unitary, which, in turn, is a consequence of the QMF property. Hence $A$ is unitary.\\
If $A$ is unitary, we check the QMF relations:

$$\frac{1}{N}\sum_{r(w)=z}m'_i(w)\cj{m'_j(w)}=\frac{1}{N}\sum_w\sum_k A_{ik}(r(w))m_k(w)\sum_l\cj{A_{jl}(r(w))m_l(w)}=$$

$$  \frac{1}{N}\sum_{k,l}A_{ik}(z)\cj{A_{jl}(z)}\sum_wm_k(w)\cj{ m_l(w) }=\sum_{k,l}A_{jk}(z)\cj{A_{jl}(z)}\delta_{kl}=\delta_{ij}$$

Hence $(m'_i)_{i=0}^{N-1}$ is a QMF basis.\\
The fact that the two correspondences are inverse to each other follows from the next computation:
$$\sum_jA_{ij}(r(z))m_j(z)=\sum_j\left( \frac{1}{N} \sum_{r(w)=r(z)}m_i'(w)  \cj{m_j}(w) \right) m_j(z)=\sum_{r(w)=r(z)}m_i'(w)\cdot\frac{1}{N}\sum_j\cj{m_j}(w)m_j(z)$$
$$=\sum_{r(w)=r(z)}m_i'(w)\delta_{wz}=m_i'(z)$$
\end{proof}

\begin{remark}
Note that the equation (1) can be reformulated as 
$A_{ij}(r(z))=\mathbb{E}(m_i'\cj{m}_j)$.
The conditional expectation $\mathbb{E}$ can be regarded as a $L^{\infty}(X,\mu)$-valued inner product 
$\ip{f}{g}_{L^{\infty}(X,\mu)}=\mathbb{E}(f\cj{g})$ for $f$, $g\in L^{\infty}(X,\mu)$.
The QMF basis condition is equivalent to the orthogonality of $(m_i)_{i=0}^{N-1}$ with respect to this inner product.  Since the dimension of $L^{\infty}(X,\mu)$ as a module over $\mathbb{E}(L^{\infty}(X,\mu))=L^{\infty}(X, r^{-1}(\B), \mu)$
is $N$, the completeness is automatic, so $(m_i)_{i=0}^{N-1}$ is an orthonormal basis for this inner product.  Thus $A\circ r$ is the change of base matrix from $(m_i)$ to $(m'_i)$. Equation (2) can be understood in the sense that a unitary matrix maps orthonormal bases into orthonormal bases.
 
\end{remark}

\section{Orthonormal bases generated by Cuntz algebras}

Next, we present the central result of our paper. It gives a general criterion for a family generated by the Cuntz isometries to be an orthonormal basis. 
\begin{theorem}\label{diez}
Let $\H$ be a Hilbert space and $(S_i)_{i=0}^{N-1}$ be a representation of the Cuntz algebra $\mathcal{O}_N$. Let $\mathcal{E}$ be an orthonormal set in $\H$ and $f:X\rightarrow \H$ a norm continuous function on a topological space $X$ with the following properties:

\begin{enumerate}

\item $\mathcal{E}=\cup_{i=0}^{N-1} S_i\mathcal{E}$.

\item $\Span\{ f(t): t\in X\}=\H$ and $\norm{f(t)}=1$, for all $t\in X$.

\item There exist functions $\mathfrak{m}_i: X\rightarrow\C$, $g_i:X\rightarrow X$, $i=0,\dots, N-1$ such that 
\beq\label{ceq1}
S_i^*f(t)=\mathfrak{m}_i(t)f(g_i(t)), \quad t\in X.
\eeq

\item There exist $c_0\in X$ such that $f(c_0)\in \Span \mathcal{E}$.

\item The only function $h\in\mathcal{C}(X)$ with $h\geq 0$, $h(c)=1$, $\forall$ $c\in\{ x\in X:f(x)\in \Span\mathcal{E}\}$, and 
\beq\label{ceq2}
h(t)=\sum_{i=0}^{N-1}\abs{\mathfrak{m}_i(t)}^2h(g_i(t)),\quad t\in X
\eeq
are the constant functions.

\end{enumerate}
Then $\mathcal{E}$ is an orthonormal basis for $\H$.
\end{theorem}

\begin{proof}Define $$h(t):=\sum_{e\in\mathcal{E}}\abs{\ip{f(t)}{e}}^2=\norm{Pf(t)}^2,\quad t\in X$$
where $P$ is the orthogonal projection onto the closed linear span of $\mathcal{E}$.

Since $t\mapsto f(t)$ is norm continuous we get that $h$ is continuous. Clearly $h\geq 0$. Also, if $f(c)\in\Span\mathcal{E}$, then $\norm{Pf(c)}=\norm{f(c)}=1$ so $h(c)=1$. In particular, from (ii) and (iv), $h(c_0)=1$. We check (\ref{ceq2}). Since the sets $S_i\mathcal{E}$, $i=0,\dots N-1$ are mutually orthogonal, the union in (i) is disjoint. Therefore for all $t\in X$ :
$$  h(t)=\sum_{i=0}^{N-1}\sum_{e\in\mathcal{E}}\abs{\ip{f(t)}{S_ie}}^2= \sum_{i=0}^{N-1}\sum_{e\in\mathcal{E}}\abs{\ip{S_i^*f(t)}{e}}^2=
\sum_{i=0}^{N-1}\abs{\mathfrak{m}_i(t)}^2\sum_{e\in\mathcal{E}}\abs{\ip{f(g_i(t))}{e}}^2=$$
$$=\sum_{i=0}^{N-1}\abs{\mathfrak{m}_i(t)}^2h(g_i(t))$$
By (v), $h$ is constant and, since $h(c_0)=1$, $h(t)=1$ for all $t\in X$. Then $\norm{Pf(t)}=1$ for all $t\in X$. Since $\norm{f(t)}=1$ it follows that $f(t)\in\mbox{span} \mathcal{E}$ for all $t\in X$. But the vectors $f(t)$ span $\H$ so $\Span\mathcal{E}=\H$ and $\mathcal{E}$ is an orthonormal basis.

\end{proof}

\begin{remark}
The operators of the form 
$$Rh(t)=\sum_{i=0}^{N-1}|\mathfrak m_i(t)|^2h(g_i(t)),\quad t\in X, h\in C(X),$$
that appear in \eqref{ceq2}, are sometimes called Ruelle operators or transfer operators, see e.g. \cite{Bal00}.
\end{remark}

\subsection{Piecewise exponential bases on fractals}
We apply Theorem \ref{diez} to the setting of Example \ref{ex2.10}, in dimension $d=1$ for affine iterated function systems, when the set $\frac1RB$ has a spectrum $L$. 

\begin{definition}Let $L$ in $\br$, $|L|=N$, $R>1$ such that $L$ is a spectrum for the set $\frac{1}{R}B$. We say that $c\in\br$ is an {\it{extreme cycle point}} for $(B,L)$ if there exists
$l_0,l_1,\dots,l_{p-1}$ in $L$ such that, if $c_0=c$, $c_1=\frac{c_0+l_0}{R}, c_2=\frac{c_1+l_1}{R}\dots c_{p-1}=\frac{c_{p-2}+l_{p-2}}{R}$ then $\frac{c_{p-1}+l_{p-1}}{R}=c _0$, and 
$\abs{m_B(c_i)}=1$ for $i=0,\dots,p-1$ where $$m_B(x)=\frac{1}{N}\sum_{b\in B}e^{2\pi ibx}\quad x\in\br.$$

\end{definition}

\begin{definition}
We denote by $L^*$ the set of all finite words with digits in $L$, including the empty word. 
For $l\in L$  let $S_l$ be given as in (\ref{e17}) where $m_l$ is replaced by the exponential $e_l$. If $w=l_1l_2\dots l_n\in L^*$ then by $S_w$ we denote the composition $S_{l_1}S_{l_2}\dots S_{l_n}$.
\end{definition}

\begin{theorem}\label{th3.4} Let $B\subset\br$, $0\in B$, $|B|=N$, $R>1$ and let $\mu_B$ be the invariant measure associated to the IFS $\tau_b(x)=R^{-1}(x+b)$, $b\in B$. Assume that the IFS has no overlap and that the set 
$\frac{1}{R}B$ has a spectrum $L\subset \br$, $0\in L$. Then the set 
$$\mathcal{E}(L)=\{S_we_{-c}:c \mbox{ is an extreme cycle point for }(B,L), w\in L^* \}$$
is an orthonormal basis in $L^2(\mu_B)$. Some of the vectors in $\mathcal E(L)$ are repeated but we count them only once.  
\end{theorem}

\begin{proof}Let $c$ be an extreme cycle point. Then $\abs{m_B(c)}=1$. Using the fact that we have equality in the triangle inequality $(1=\abs{m_B(c)}\leq \frac{1}{N}\sum_{b\in B}\abs{e^{2\pi ibc}}=1)$ 
, and since $0\in B$, we get that $e^{2\pi ibc}=1 $ so $bc\in \bz$ for all $b\in B$. Also there exists another extreme cycle point $d$ and $l\in L$ such that $\frac{d+l}{R}=c$. Then we have:
$S_le_{-c}(x)=e^{2\pi ilx}e^{2\pi i(Rx-b)(-c)}$, if $x\in\tau_b(X_B)$. Since $bc\in\bz$ and $R(-c)+l=-d$, we obtain
\beq\label{cex}
S_le_{-c}=e_{-d}
\eeq
We use this property to show that the vectors $S_we_{-c}$, $S_{w'}e_{-c'}$ are either equal or orthogonal for $w, w'$ in $L^*$ and $c, c'$ extreme cycle points for $(B,L)$. Using (\ref{cex}), we
can append some letters at the end of $w$ and $w'$ suh that the new words have the same length:
$$ S_we_{-c}=S_{w\alpha}e_{-d},\quad  S_{w'}e_{-c'}=S_{w'\beta}e_{-d'},\quad \abs{w\alpha}=\abs{w'\beta}\quad\mbox{where }d,d'\mbox{ are cycle points.}$$
Moreover, repeating the letters for the cycle points $d$ and $d'$ as many times as we want, we can assume that $\alpha$ ends in a repetition of the letters associated to $d$ and similarly 
for $\beta$ and $d'$. But, since $ \abs{w\alpha}=\abs{w'\beta}$, the Cuntz relations imply that $S_{w\alpha}e_{-d}\perp S_{w'\beta}e_{-d'}$ or $w\alpha=w'\beta$. Assume $\abs{w}\leq\abs{w'}$. Then 
$\alpha=w''\beta$ for some word $w''$. Then  $S_{w\alpha}e_{-d}\perp S_{w'\beta}e_{-d}$ iff  $S_{\alpha}e_{-d}\perp S_{w''\beta}e_{-d'}$. Also, $\alpha$ consists of repetitions of the digits of the cycle associated to 
$d$ and similarly for $d'$. So $S_{\alpha}e_{-d}=e_{-f}$, $S_{w''\beta}e_{-d'}=e_{-f'}$, and all points $d, d', f, f', c, c'$ all belong to the same cycle. So the only case when  $S_we_{-c}$ is not orthogonal to $S_{w'}e_{-c'}$ is when they are equal.

Next we check that the hypotheses of Theorem \ref{diez} are satisfied. We let $f(t)=e_{-t}\in L^2(\mu_B)$. To check (i) we just to have to see that $e_{-c}\in\cup_{l\in L}S_l\mathcal{E}(L)$. But this follows from (\ref{cex}). Requirement (ii) is clear. For (iii) we compute 
$$S_l^*e_{-t}(x)=\frac{1}{N}\sum_{b\in B}e^{-2\pi il\cdot\frac{1}{R}(x+b)}e^{-2\pi it\cdot\frac{1}{R}(x+b)}=e^{-2\pi x\cdot\frac{1}{R}(t+l)}\frac{1}{N}\sum_{b\in B}e^{-2\pi ib(\frac{t+l}{R})}= $$
$$=\cj{m_B}\left(\frac{t+l}{R}\right)e_{-\frac{t+l}{R}}(x)$$
So (iii) is satisfied with $\mathfrak{m}_l(t)=\cj{m_B}(\frac{t+l}{R})$, $g_l(t)=\frac{t+l}{R}$.

For (iv) take $c_0=-c$ for any extreme cycle point ( $0$ is always one). 
For (v), take $h$ continuous on $\br$ , $0\leq h\leq 1$, $h(c)=1$ for all $c$ with $e_{-c}\in \Span \mathcal{E}(L)$, and 
$$h(t)=\sum_{l\in L}\left|m_B\left(\frac{t+l}{R} \right)\right|^2h\left( \frac{t+L}{R}\right):=Rh(t)$$
In particular, we have $h(c)=1$ for every extreme cycle point $c$. Assume $h\not\equiv 1$. First we will restrict our attention to $t\in I:=[a,b]$ with $a\leq \frac{\mbox{min}L}{R-1}$, 
$b\geq \frac{\mbox{max}L}{R-1}$, and note that $g_l(I)\subset I$ for all $l\in L$. Let $m=\mbox{min}_{t\in I}h(t)$. Then let $h'=h-m$, assume $m<1$. Then $Rh'(t)=h'(t)$ for all 
$t\in\br$, $h'$ has a zero in $I$ and $h\geq 0$ on $I$, $h'(z_0)=0$. But this implies that $\abs{m_B(g_l(z_0))}^2h'(g_l(z_0))=0$ for all $l\in L$. Since $\sum_{l\in L}\abs{m_B(g_l(z_0))}^2=1$, it follows 
that for one of the $l_0\in L$ we have $h'(g_{l_0}(z_0))=0$. By induction, we can find $z_n=g_{l_{n-1}}\cdots g_{l_0}z_0$ such that $h'(z_n)=0$. We prove that $z_0$ is a cycle point. Suppose not. Since $m_B$ has finitely many zeros, for $n$ large enough $g_{\alpha_k}\cdots g_{\alpha_1}z_n$ is not a zero for $m_B$, for any choice of digits $\alpha_1,\dots,\alpha_k$ in $L$. But then, by using the same argument as above 
we get that $h'(g_{\alpha_k}\cdots g_{\alpha_1}z_n)=0$ for any $\alpha_1,\dots,\alpha_k\in L$. The points $\{g_{\alpha_k}\cdots g_{\alpha_1}z_n: \alpha_1,...\alpha_k\in L,k\in \N\}$ are dense 
in the attractor $X_L$ of the IFS $\{g_l\}_{l\in L}$, thus $h'$ is constant $0$ on $X_L$. But the extreme cycle points $c$ are in $X_L$ and since $h(c)=1$ we have $0=h'(c)=1-m$, so $m=1$. Thus $h=1$ on $I$. Since we can let $a\rightarrow-\infty$ and $b\rightarrow\infty$ we obtain that $h\equiv 1$.

\end{proof}

\begin{remark}\label{rpex}The functions in $\mathcal{E}(L)$ are piecewise exponential. The formula for $S_{l_1...l_n}e_{-c}$ is 
\beq\label{pex}
S_{l_1...l_n}e_{-c}(x)=e^{\alpha(b,l,c)}\cdot e_{l_1+Rl_2+...+R^{n-1}l_{n-1}+R^n(-c) }(x)
\eeq
where $\alpha(b,l,c)=-[b_1l_2+(Rb_1+b_2)l_3+...+(R^{n-2}b_1+...+b_{n-1})l_n]+(R^{n-1}b_1+...+b_n)\cdot c$ if $x\in\tau_{b_1}...\tau_{b_n}X_B$ . We have
$$S_{l_1}...S_{l_n}e_{-c}(x)=e_{l_1}(x)e_{l_2}(rx)...e_{l_n}(r^{n-1}x)e_c(r^nx)$$
If $x\in\tau_{b_1}...\tau_{b_n}X_B$ then $rx\in\tau_{b_2}...\tau_{b_n}X_B$, $r^{n-1}x\in\tau_{b_n}X_B$. So 
\begin{align*}
rx&=Rx-b_1\\
r^2x&=Rrx-b_2=R^2x-Rb_1- b_2\\
&\vdots\\
r^{n-1}x&=R^{n-1}x-R^{n-2}b_1-...-Rb_{n-2}-b_{n-1}\\
r^nx&=R^nx-R^{n-1}b_1-R^{n-2}b_2-...-Rb_{n-1}-b_n.
\end{align*}
The rest follows from a direct computation.
 
\end{remark}

\begin{corollary}In the hypothesis of Theorem \ref{diez}, if in addition $B, L\subset \Z$ and $R\in \bz$, then there exists a set $\Lambda$ such that $\{ e_{\lambda} : \lambda\in\Lambda\}$ is an 
orthonormal basis for $L^2(\mu_B)$.

\end{corollary}
\begin{proof} If everything is an integer then, it follows from Remark \ref{rpex} that $S_we_{-c}$ is an exponential function for all $w$ and extreme cycle points $c$. Note that, as in the proof of Theorem \ref{diez}, 
$bc\in\Z$ for all $b\in B$.

\end{proof}

\begin{example}\label{ex3.7} We consider the IFS that generates the middle third Cantor set: $R=3$, $B=\{0,2\}$. The set $\frac{1}{3}\{0,2\}$ has spectrum $L=\{0,3/4\}$. We look for the extreme cycle points for 
$(B,L)$. \\We need $\abs{m_B(-c)}=1$ so $\abs{\frac{1+e^{2\pi i2c}}{2}}=1$, therefore $c\in\frac12\Z$. Also $c$ has to be a cycle for the IFS $g_0(x)=x/3$, $g_{3/4}(x)=\frac{x+3/4}{3}$ so $0\leq c\leq \frac{3/4}{3-1}=3/8$. Thus, the only extreme cycle is $\{0\}$. By Theorem \ref{diez} $\mathcal{E}=\{S_w1: w\in\{0,3/4\}^*\}$ is an orthonormal basis for $L^2(\mu_B)$. Note also that the numbers $e^{ 2\pi i\alpha(b,l,c)}$ in formula (\ref{pex}) are $\pm 1$ because $2\pi i B\cdot L\subset \pi i\Z$.
\end{example}

\subsection{Walsh bases} In the following, we will focus on the unit interval, which can be regarded as the attractor of a simple IFS and we use step functions for the QMF basis to generate Walsh-type bases for $L^2[0,1]$.

\begin{example}The interval $[0,1]$ is the attractor of the IFS $\tau_0x=\frac{x}{2}$, $\tau_1x=\frac{x+1}{2}$, and the invariant measure is the Lebesgue measure on $[0,1]$. The map $r$ defined in Example \ref{ex2} is $rx=2x\mbox{mod}1$. Let $m_0=1$, $m_1=\chi_{[0,1/2)}-\chi_{[1/2,1)}$. It is easy to see that $\{m_0, m_1\}$ is a QMF basis. Therefore $S_0$, $S_1$ defined as in Proposition \ref{pr11}
form a representation of the Cuntz algebra $\mathcal{O}_2$. 
\end{example}

\begin{proposition}
The set $\mathcal{E}:=\{S_w1: w\in\{0,1\}^*\}$ is an orthonormal basis for $L^2[0,1]$, the Walsh basis.
\end{proposition}
\begin{proof}
We check the conditions in Theorem \ref{diez}. To see that (i) holds note that $S_01=1$. Define $f(t)=e_t$, $t\in\br$. (ii) is clear. For (iii) we compute 
$$ S_1^*e_t(x)=\frac{1}{2}(e^{2\pi it\cdot x/2}+e^{2\pi it\cdot (x+1)/2})=e^{2\pi it\cdot x/2}\frac{1}{2}(1+e^{2\pi it/2})$$
$$ S_1^*e_t(x)=\frac{1}{2}(e^{2\pi it\cdot x/2}-e^{2\pi it\cdot (x+1)/2})=e^{2\pi it\cdot x/2}\frac{1}{2}(1-e^{2\pi it/2})$$
Thus (iii) holds with $\mathfrak{m}_0(t)=\frac{1}{2}(1+e^{2\pi it/2})$,  $\mathfrak{m}_1(t)=\frac{1}{2}(1-e^{2\pi it/2})$, $g_0(t)=g_1(t)=\frac{t}{2}$. 
Since $e_0=1$ it follows that (iv) holds.

For (v) take $h$ continuous on $\br$, $0\leq h\leq 1$, $h(c)=1$ for all $c\in\br$ with $e_t\in\Span\mathcal{E}$, in particular $h(0)=1$ and 
$$h(t)= \left| \frac{1}{2}(1+e^{2\pi it/2}) \right|^2h(t/2)+\left| \frac{1}{2}(1-e^{2\pi it/2}) \right|^2h(t/2)=h(t/2)$$
Then $h(t)=h(t/2^n)$ for all $t\in\br$, $n\in\N$. Letting $n\rightarrow\infty$ and using the continuity of $h$, we get $h(t)=h(0)=1$ for all $t\in\br$. Since all conditions hold,
we get that $\mathcal{E}$ is an orthonormal basis. That $\mathcal{E}$ is actually the Walsh basis follows from the following calculations: for $\abs{w}=n$ in $\{0,1\}^*$ let $n=\sum_ix_i2^i$ be the base $2$ expansion
of $n$. Because $S_0f=f\circ r$, $S_1f=m_1f\circ r$ and $m_0\equiv 1$ we obtain the following decomposition:
$$S_w1(x)=m_1(r^{i_1}x)\cdot m_1(r^{i_2}x)\cdots m_1(r^{i_k}x),\quad\mbox{where } i_1,i_2,\dots ,i_k\mbox{ correspond to those }i\mbox{ with }x_i=1.$$
Also $m_1(r^ix)=m_1(2^ix\mbox{mod}i)$ are the Rademacher functions and thus we obtain the Walsh basis (see e.g. \cite{ScWaSi90}). 

\end{proof}

The Walsh bases can be easily generalized by replacing the matrix $$\frac{1}{\sqrt2}\begin{pmatrix}
	1&1\\
	1&-1
\end{pmatrix}$$
which appears in the definition of the filters $m_0,m_1$, with an arbitrary unitary matrix $A$ with constant first row and by changing the scale from 2 to $N$.

\begin{theorem}\label{th3.10} Let $N\in \N$, $N\geq 2$. Let $A=[a_{ij}]$ be an $N\times N$ unitary matrix whose first row is constant $\frac{1}{\sqrt{N}}$. Consider the IFS $\tau_jx=\frac{x+j}{N}$, 
$x\in\br$, $j=0,\dots,N-1$ with the attractor $[0,1]$ and invariant measure the Lebesgue measure on $[0,1]$. Define
$$m_i(x)=\sqrt{N}\sum_{j=0}^{N-1}a_{ij}\chi_{[j/N, (j+1)/N]}(x)$$
Then $\{m_i\}_{i=0}^{N-1}$ is a QMF basis. Consider the associated representation of the Cuntz algebra $\mathcal{O}_N$. Then the set 
$\mathcal{E}:=\{ S_w1: w\in \{0,...N-1\}^* \}$ is an orthonormal basis for $L^2[0,1]$. 

\end{theorem}

\begin{proof}We check the conditions in Theorem \ref{diez}. Let $f(t)=e_t$, $t\in\br$.

To check (i) note that $S_01\equiv 1$. (ii) is clear. For (iii) we compute:
$$S_k^*e_t=\frac{1}{N}\sum_{j=0}^{N-1}\cj{m_k}(\tau_jx)e_t(\tau_jx)=\frac{1}{\sqrt{N}}\sum_{j=0}^{N-1}\cj{a_{kj}}e^{2\pi it\cdot (x+j)/N}=
e^{2\pi it\cdot x/N}\frac{1}{\sqrt{N}}\sum_{j=0}^{N-1}\cj{a_{kj}}e^{2\pi it\cdot j/N} $$
So (iii) is true with $\mathfrak{m}_k(t)=\frac{1}{\sqrt{N}}\sum_{j=0}^{N-1}\cj{a_{kj}}e^{2\pi it\cdot j/N}$ and $g_k(t)=\frac{t}{N}$.

(iv) is true with $c_0=0$. For (v) take $h\in \mathcal{C}(\br)$, $0\leq h\leq 1$, $h(c)=1$ for all $c\in\br$ with $e_c\in\Span\mathcal{E}$ ( in particular $h(0)=1$), and 
$$ h(t)=\sum_{k=0}^{N-1}\abs{\mathfrak{m}_k(t) }^2h(t/N)=h(t/N)\sum_{k=0}^{N-1}\frac{1}{N}\abs{ \sum_{j=0}^{N-1}a_{kj}e^{ -2\pi it\cdot j/N}  }^2 = h(t/N)\cdot \frac{1}{N}\norm{Av}^2$$
where $v=(e^{ -2\pi it\cdot j/N})_{j=0}^{N-1}$. Since $A$ is unitary, $\norm{Av}^2=\norm{v}^2=N$. Then $h(t)=h(t/N^n)$. Letting $n\rightarrow \infty$ and using the continuity 
of $h$ we obtain that $h(t)=1$ for all $t\in \br$. 
Thus, Theorem \ref{diez} implies that $\mathcal{E}$ is an orthonormal basis.

\end{proof}

\begin{remark} We can read the constants that appear in the step function $S_w1$ from the tensor of $A$ with itself $n$ times, where $n$ is the length of the word $w$.

Let $A$ be an $N\times N$ matrix, $B$ an $M\times M$ matrix. Then $A\otimes B$ has entries :
$$  (A\otimes  B)_{i_1+Mi_2, j_1+Mj_2} = a_{i_1j_1}b_{i_2j_2},\quad i_1, j_1=0,\dots, N-1\mbox{, }i_2,j_2=0,\dots, M-1$$
$$A\otimes B =
 \begin{pmatrix}
  Ab_{0,0} & Ab_{0,1} & \cdots & Ab_{0,M-1} \\
  Ab_{1,0} & Ab_{1,1} & \cdots & Ab_{1,M-1} \\
  \vdots  & \vdots  & \ddots & \vdots  \\
  Ab_{M-1,0} & Ab_{M-1,1} & \cdots & Ab_{M-1,M-1}
 \end{pmatrix}$$
The matrix $A^{\otimes n}$ is obtained by induction, tensoring to the left: $A^{\otimes n}=A\otimes A^{\otimes (n-1)}$.\\
Thus $A\otimes A\otimes A\otimes \dots\otimes A$, $n$ times, has entries
$$ A^{\otimes n}_{i_0+Ni_1+N^2i_2+\dots +N^{n-1}i_{n-1}, j_0+Nj_1+\dots +N^{n-1}j_{n-1} }=a_{i_0j_0}a_{i_1j_1}\dots a_{i_{n-1}j_{n-1}}$$
Now compute for $i_0,\dotsc i_{n-1}\in\{0,\dotsc,N-1\}$:
$$S_{i_0\dots i_{n-1}}1(x)=m_{i_0}(x)m_{i_1}(rx)\dots m_{i_{n-1}}(r^{n-1}x)$$
Suppose $x\in[\frac{k}{N^n}, \frac{k+1}{N^n})$, $0\leq k<N^n$ and $k=N^{n-1}j_0+N^{n-2}j_1+\dots +Nj_{n-2}+j_{n-1}$, where $0\leq j_0,\dots, j_{n-1}<N$. \\
Then $x\in [\frac{j_0}{N}, \frac{j_0+1}{N})$, $rx=(Nx)\mbox{mod}1\in[\frac{j_1}{N},\frac{j_1+1}{N}),\dots, r^{n-1}x=(N^{n-1}x)\mbox{mod}1\in [\frac{j_{n-1}}{N},\frac{j_{n-1}+1}{N})$, so
$m_{i_0}(x)=\sqrt{N}a_{i_0j_0}$, $m_{i_1}(rx)=\sqrt{N}a_{i_1j_1}$, $\dots$,$ m_{i_{n-1}}(r^{n-1}x)=\sqrt{N}a_{i_{n-1}j_{n-1}}$ hence
$$ S_{i_0\dots i_{n-1}}1(x)=\sqrt{N^n}a_{i_0j_0}\dots a_{i_{n-1}j_{n-1}}= \sqrt{N^n} A^{\otimes n}_{i_0+Ni_1+N^2i_2+\dots +N^{n-1}i_{n-1}, j0+Nj_1+\dots +N^{n-1}j_{n-1}  }$$

\end{remark}

\begin{figure}[htb]
\label{F:graphs1}
  \begin{center}
    \includegraphics[width=1.2in]{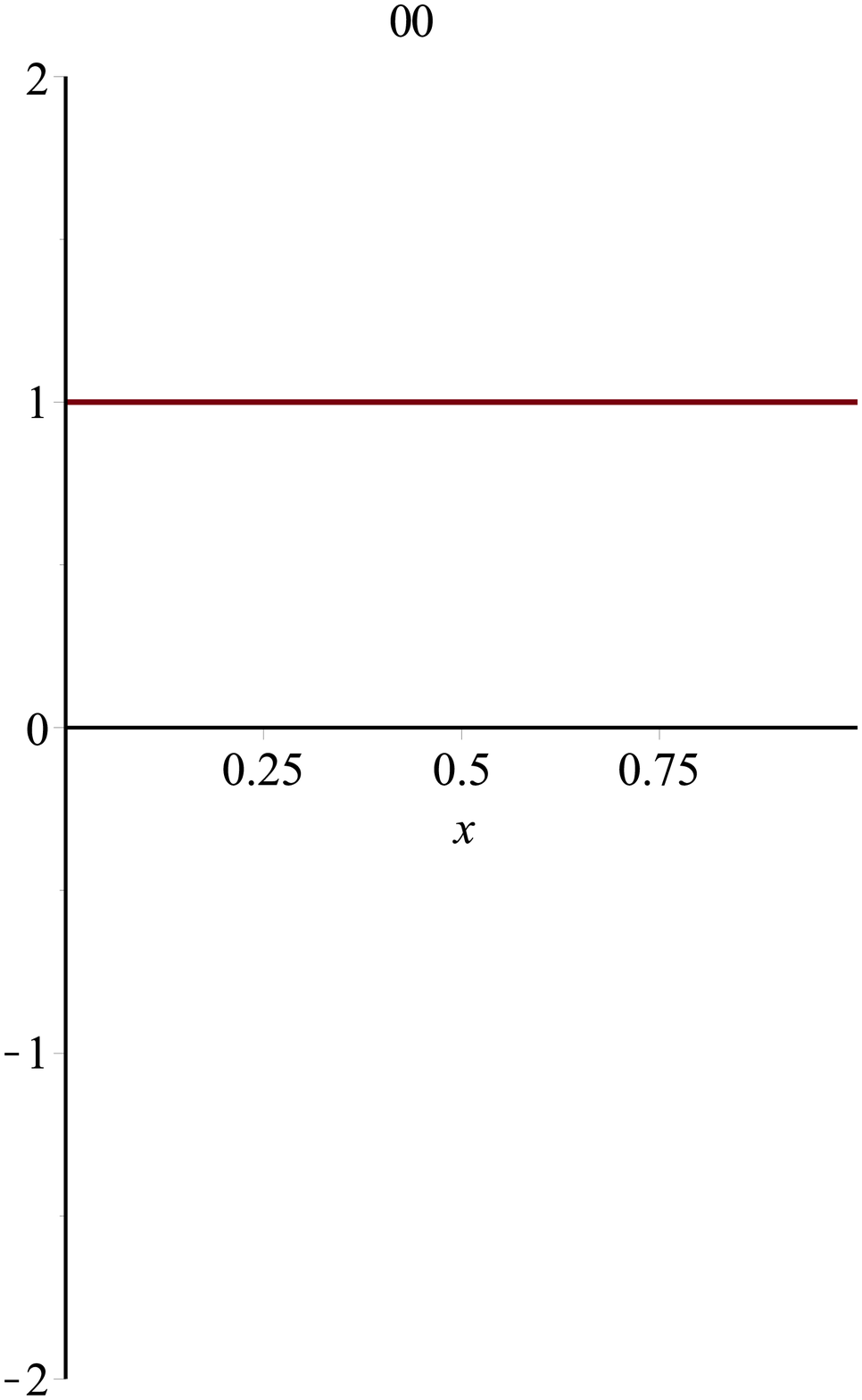}
    \includegraphics[width=1.2in]{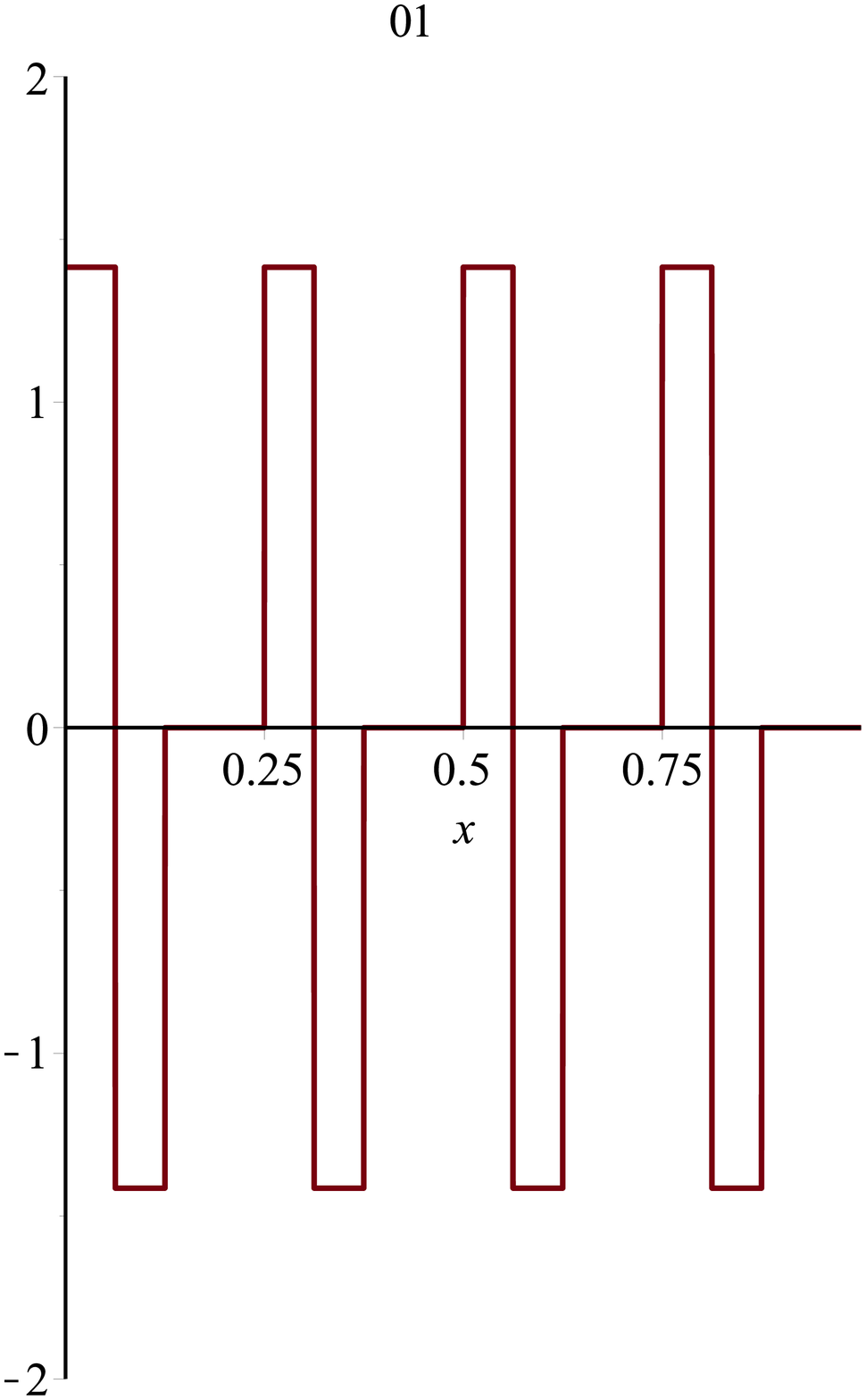}
    \includegraphics[width=1.2in]{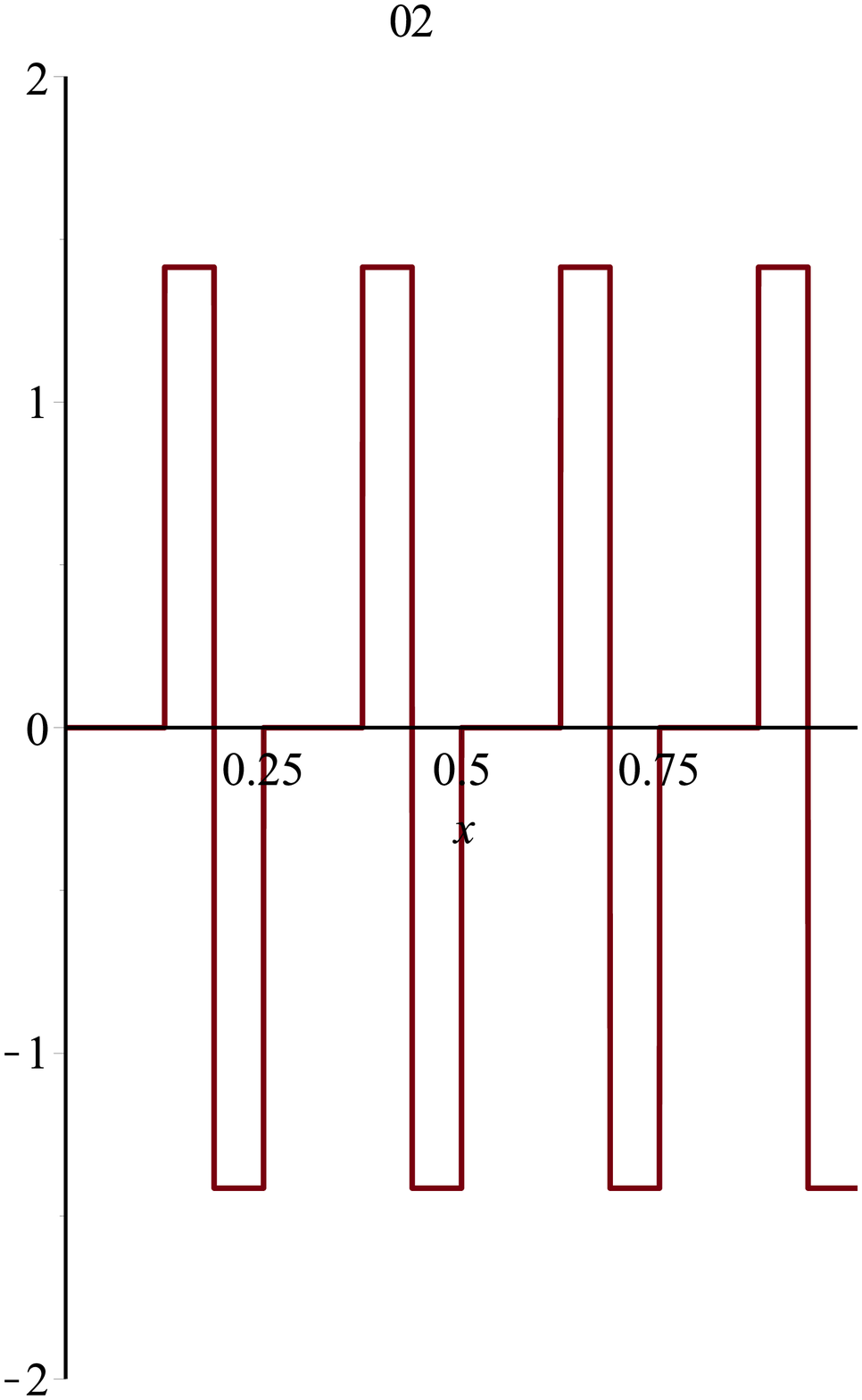}
    \includegraphics[width=1.2in]{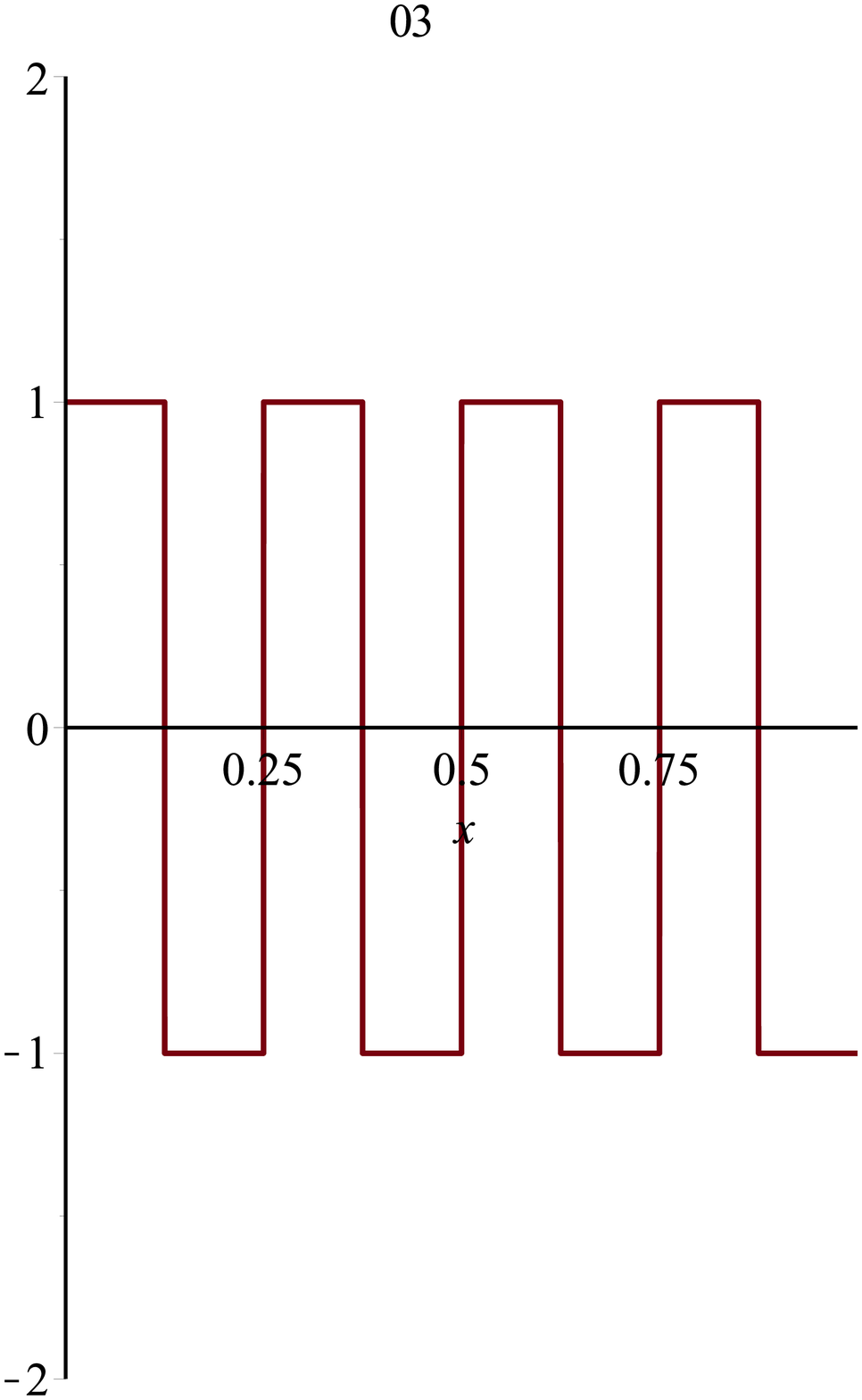}
    \newline
  \end{center}
  %\caption{Walsh functions $S_w1$ for words $w$ of length 2.}
\end{figure}
%\clearpage
\begin{figure}[htb]
\label{F:graphs2}
  \begin{center}
    %\newline
    \includegraphics[width=1.2in]{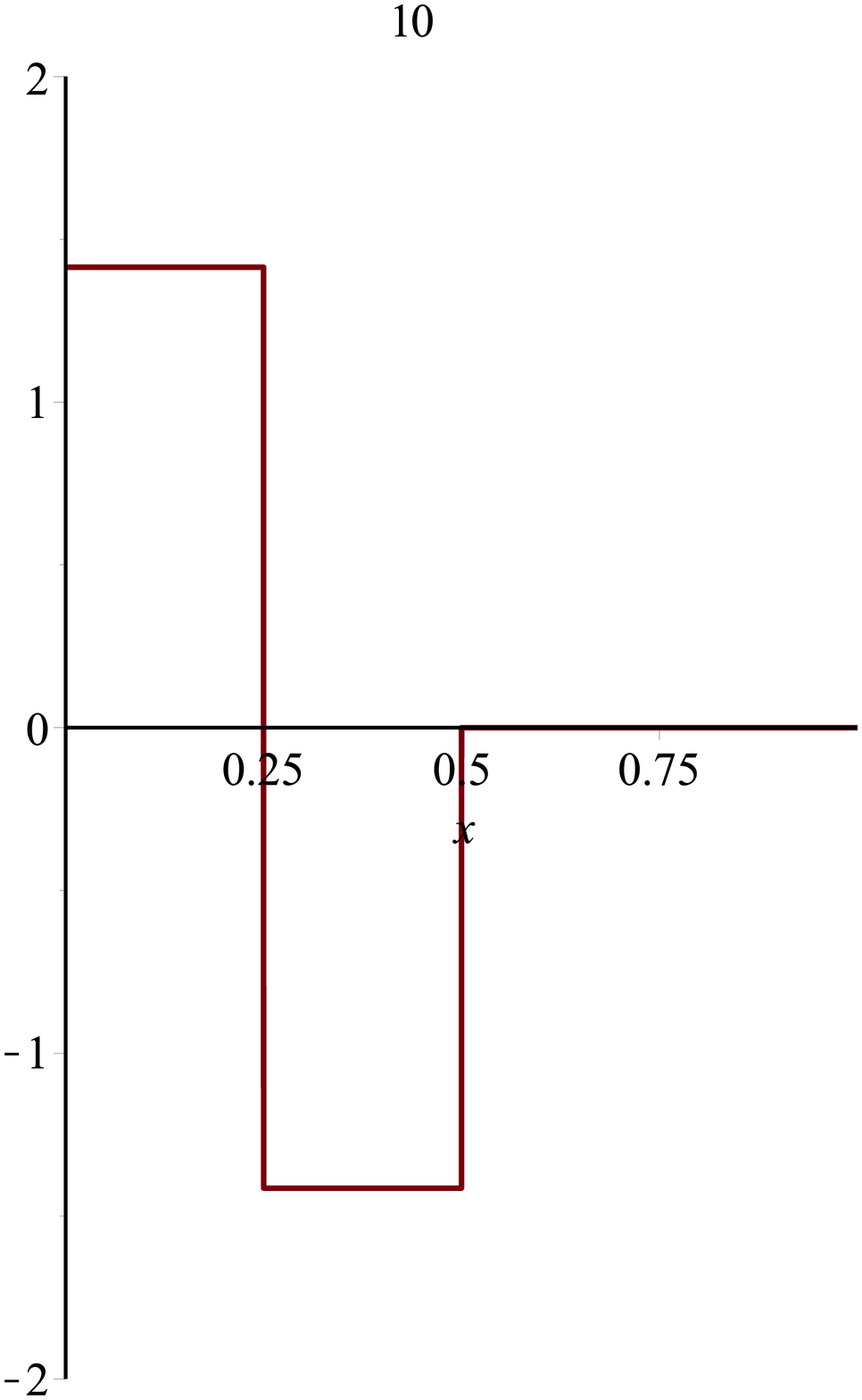}
    \includegraphics[width=1.2in]{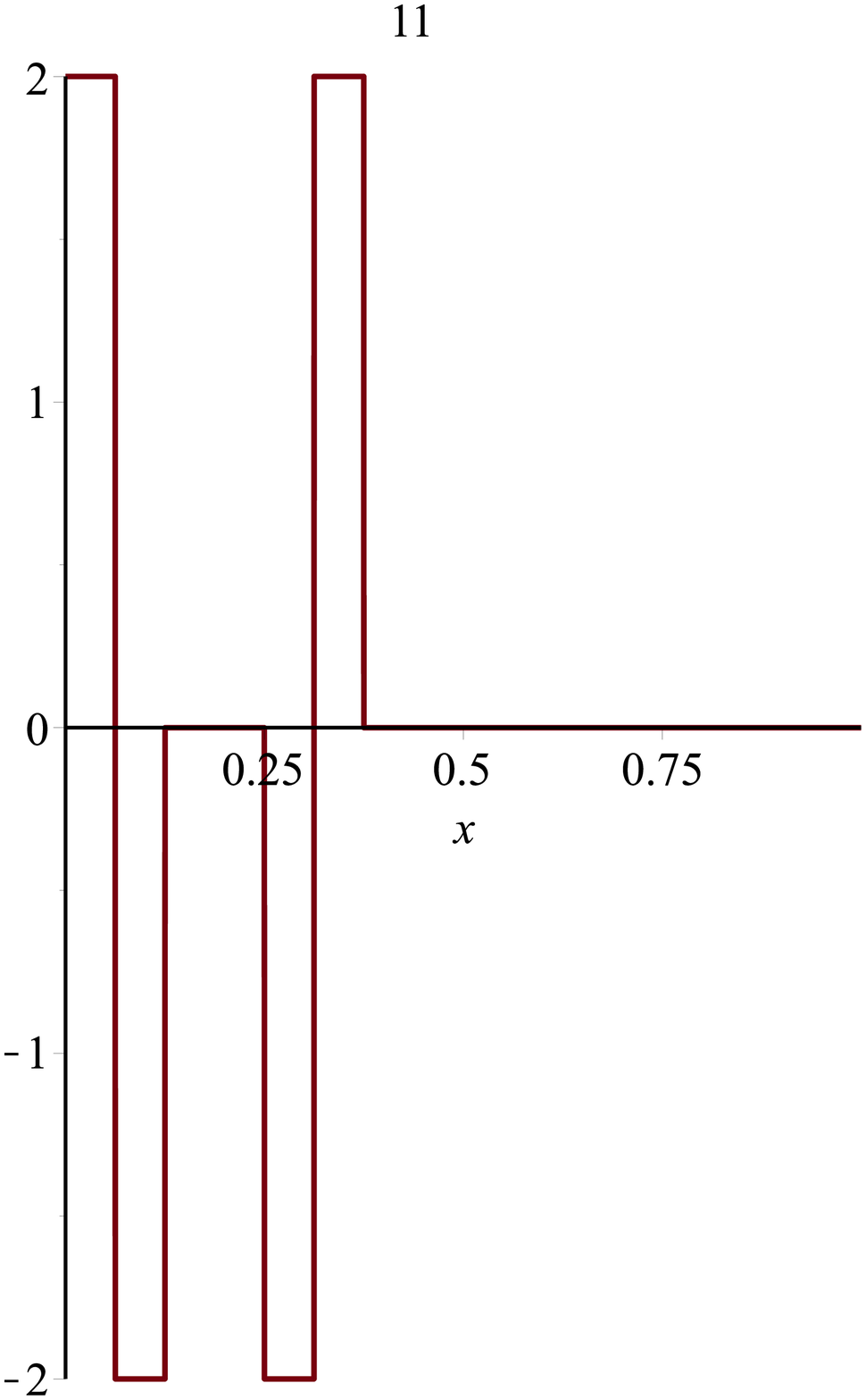}
    \includegraphics[width=1.2in]{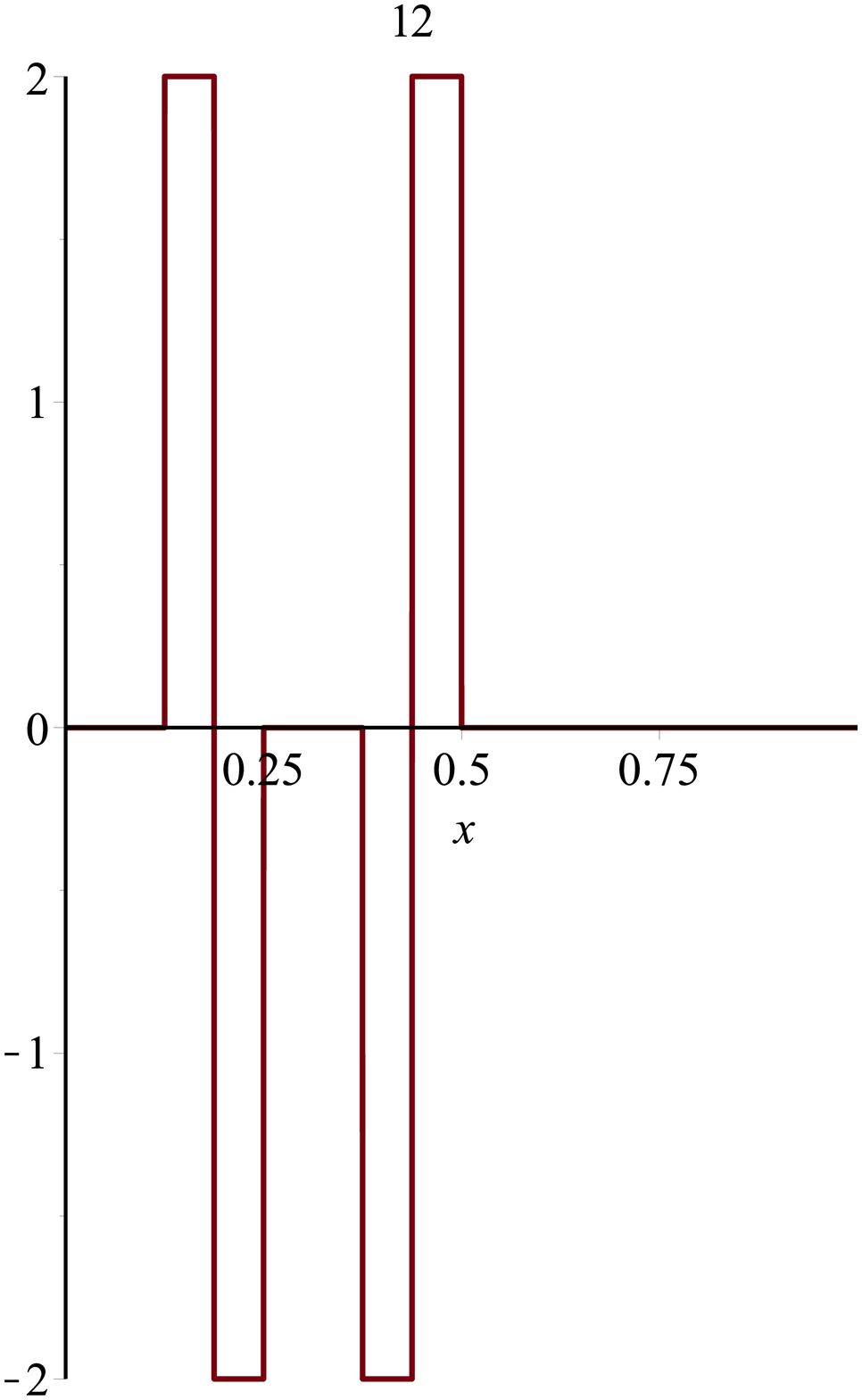}
    \includegraphics[width=1.2in]{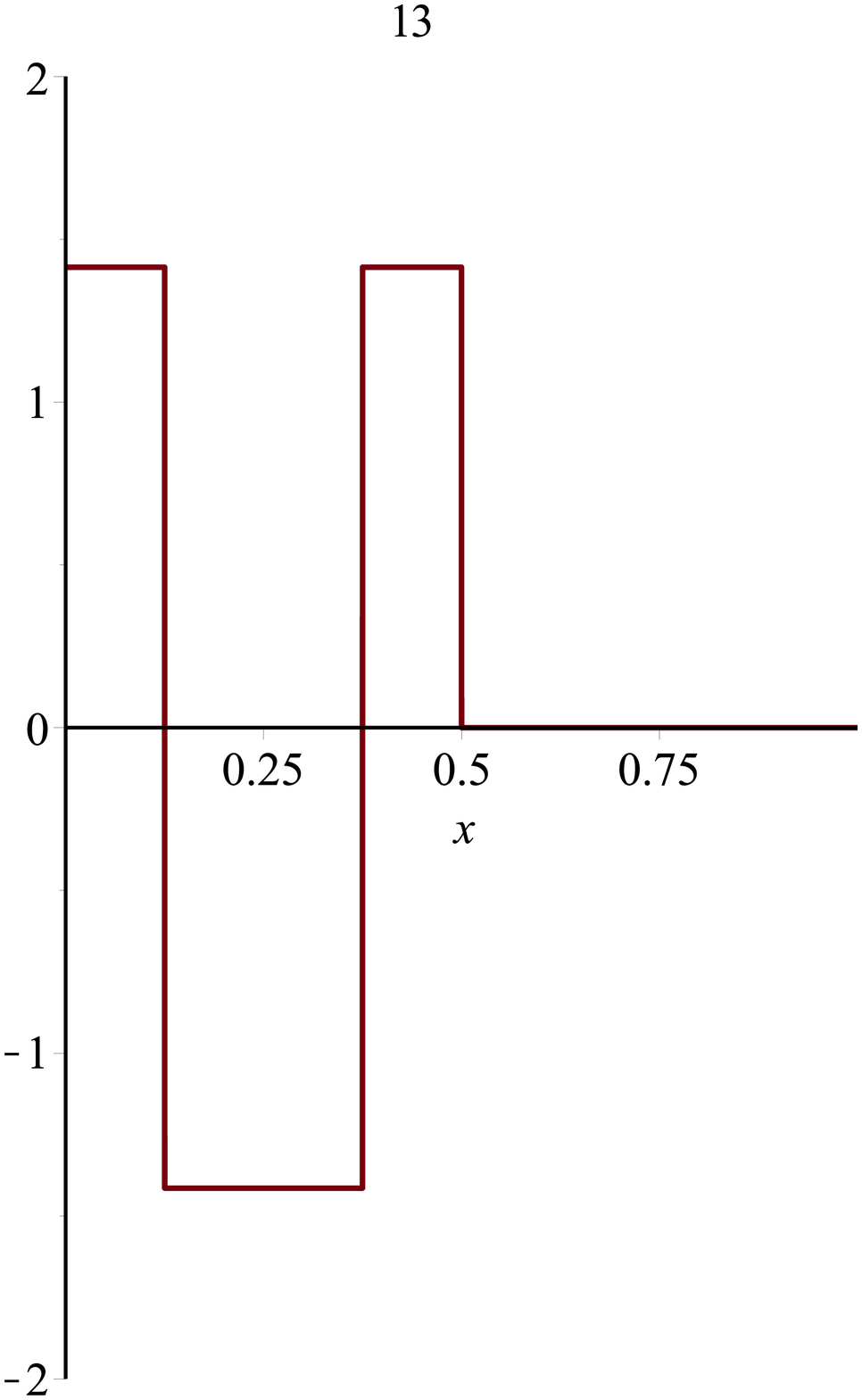}
    \newline
    \includegraphics[width=1.2in]{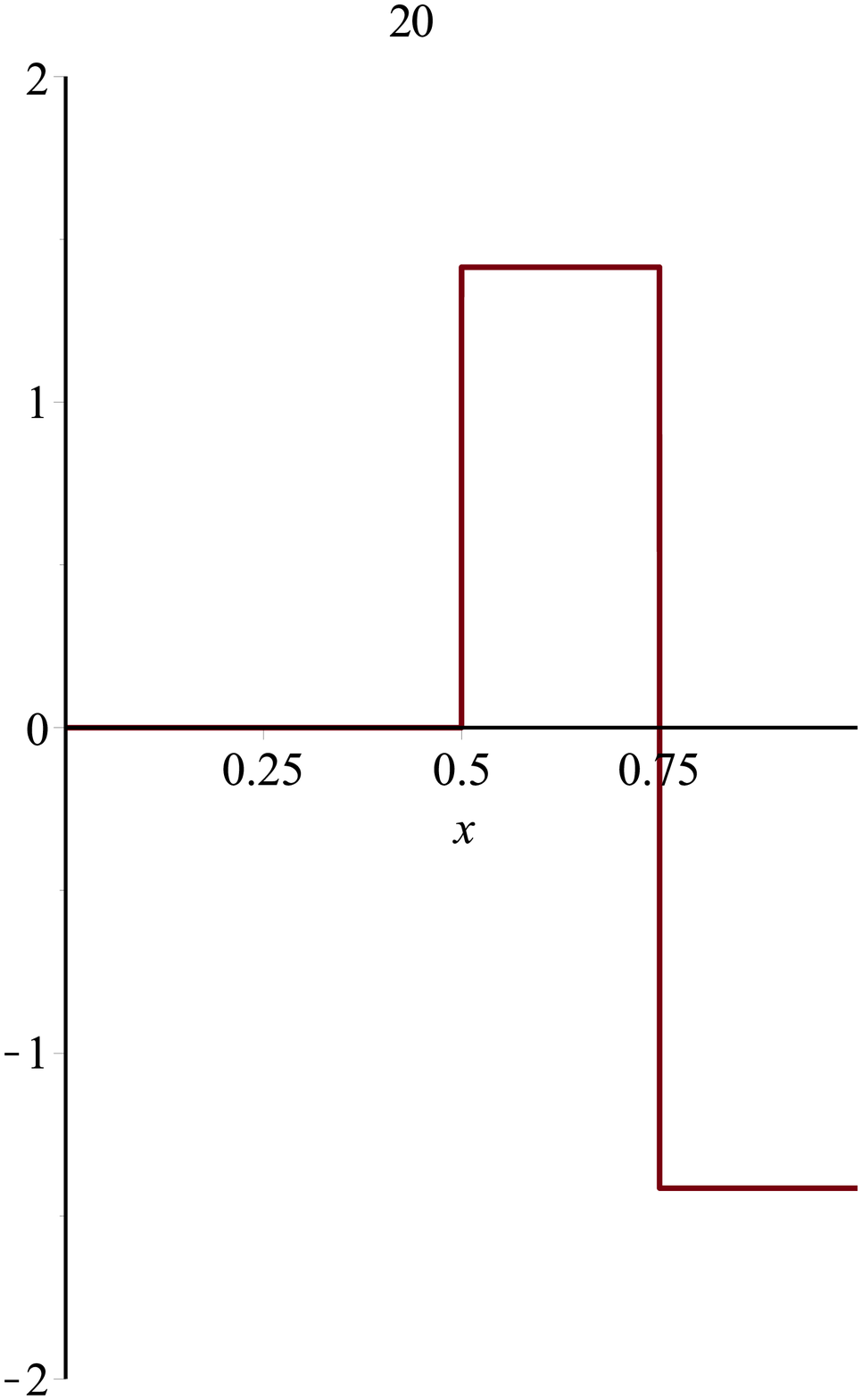}
    \includegraphics[width=1.2in]{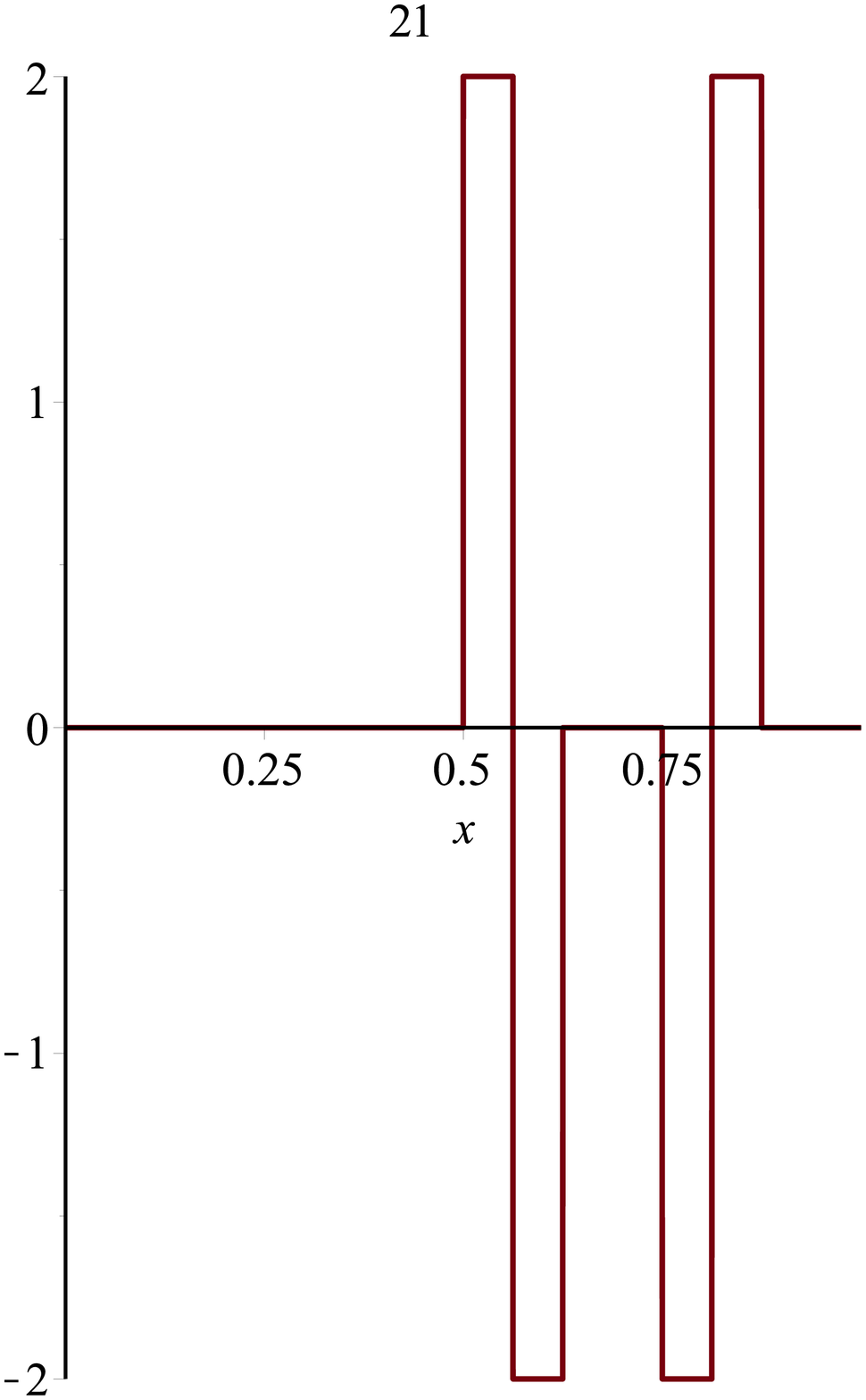}
    \includegraphics[width=1.2in]{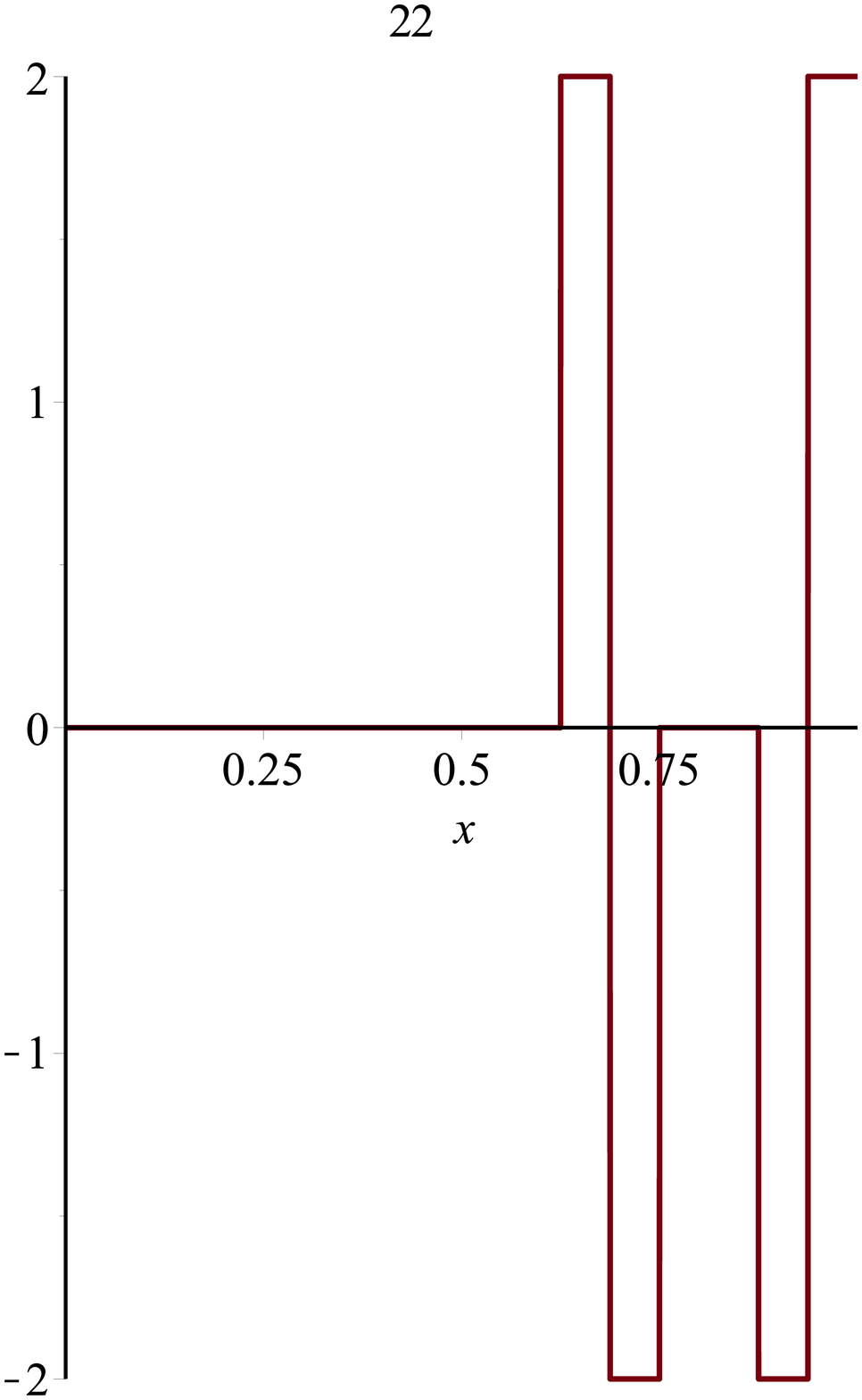}
    \includegraphics[width=1.2in]{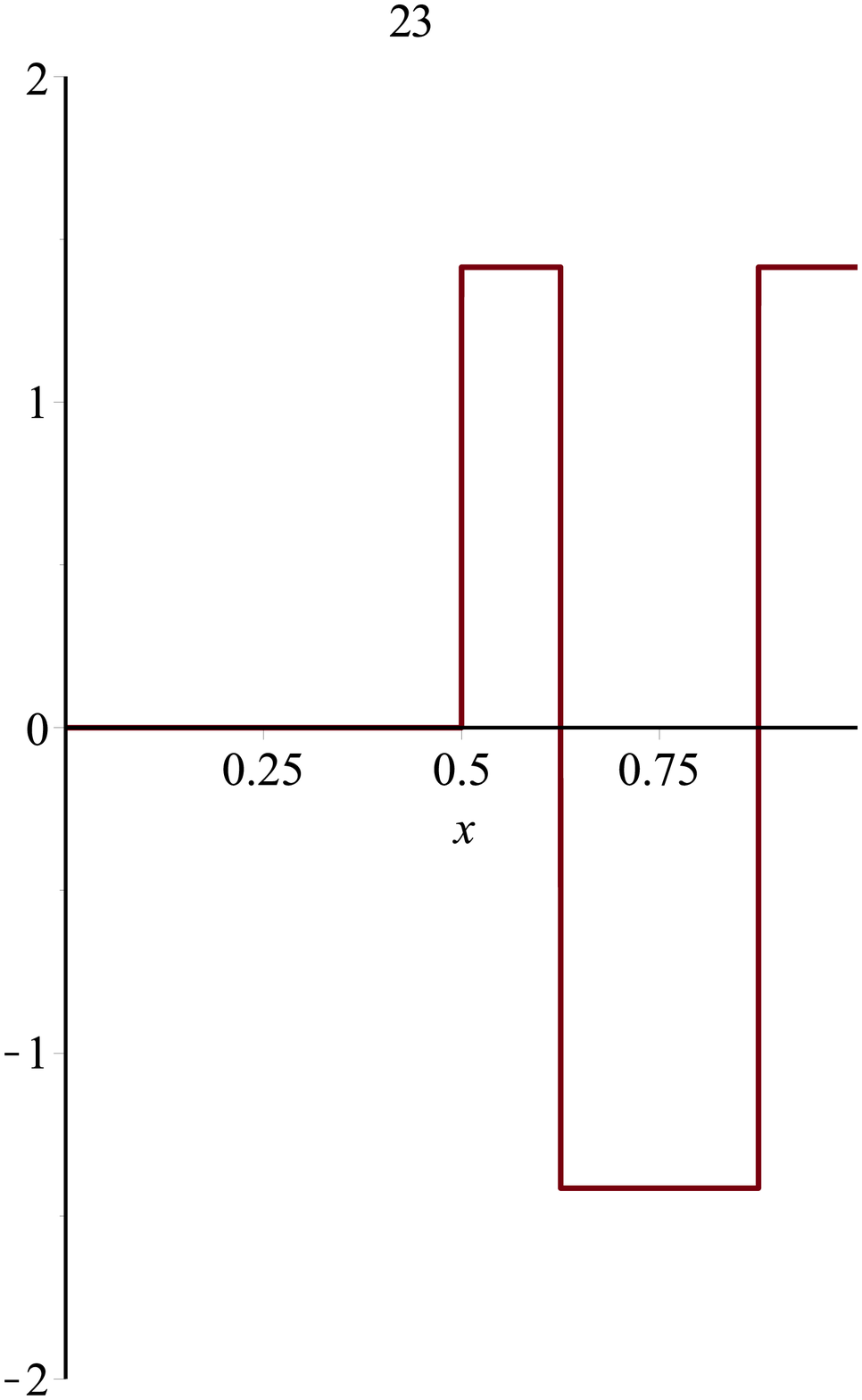}
    \newline
    \includegraphics[width=1.2in]{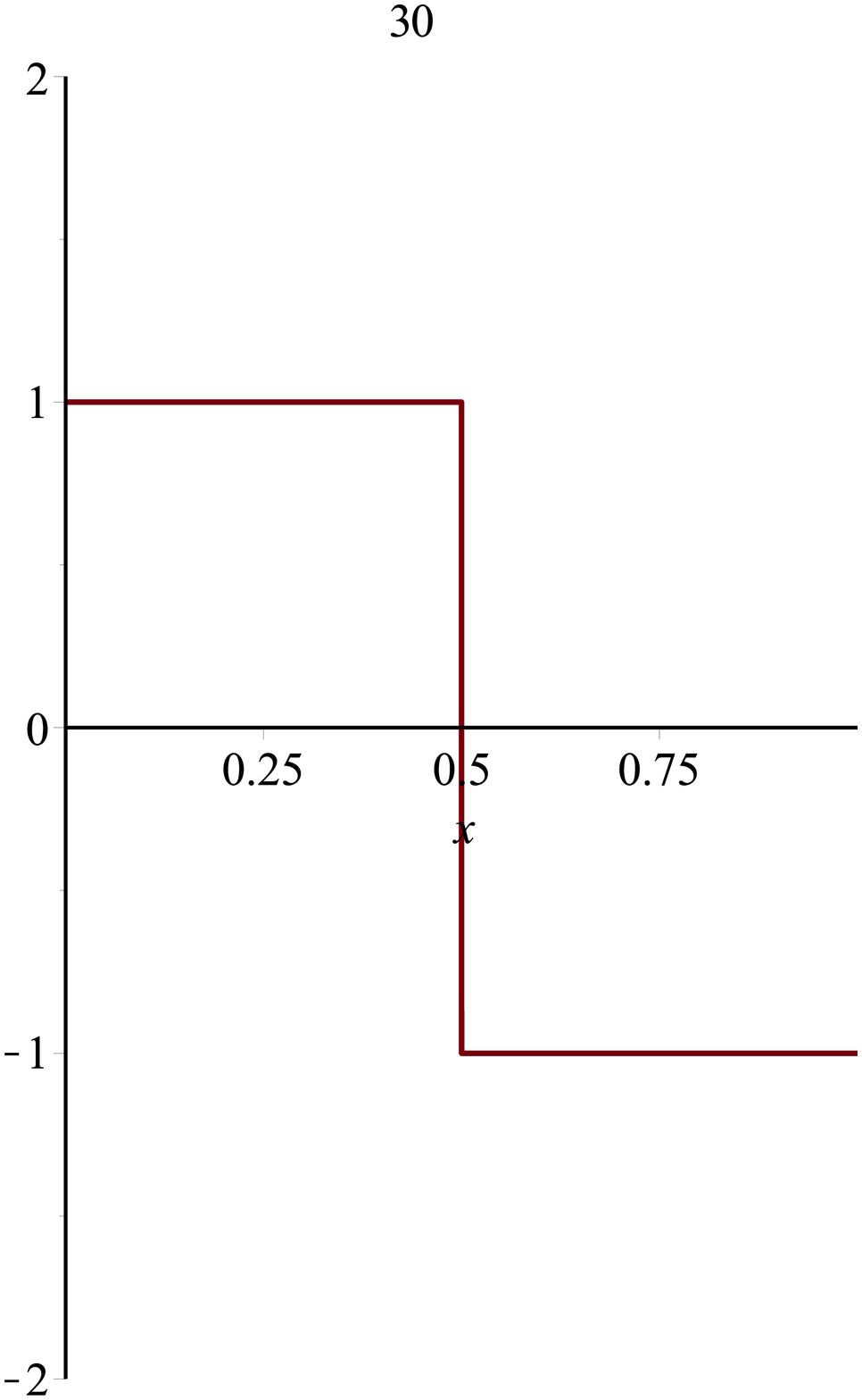}
    \includegraphics[width=1.2in]{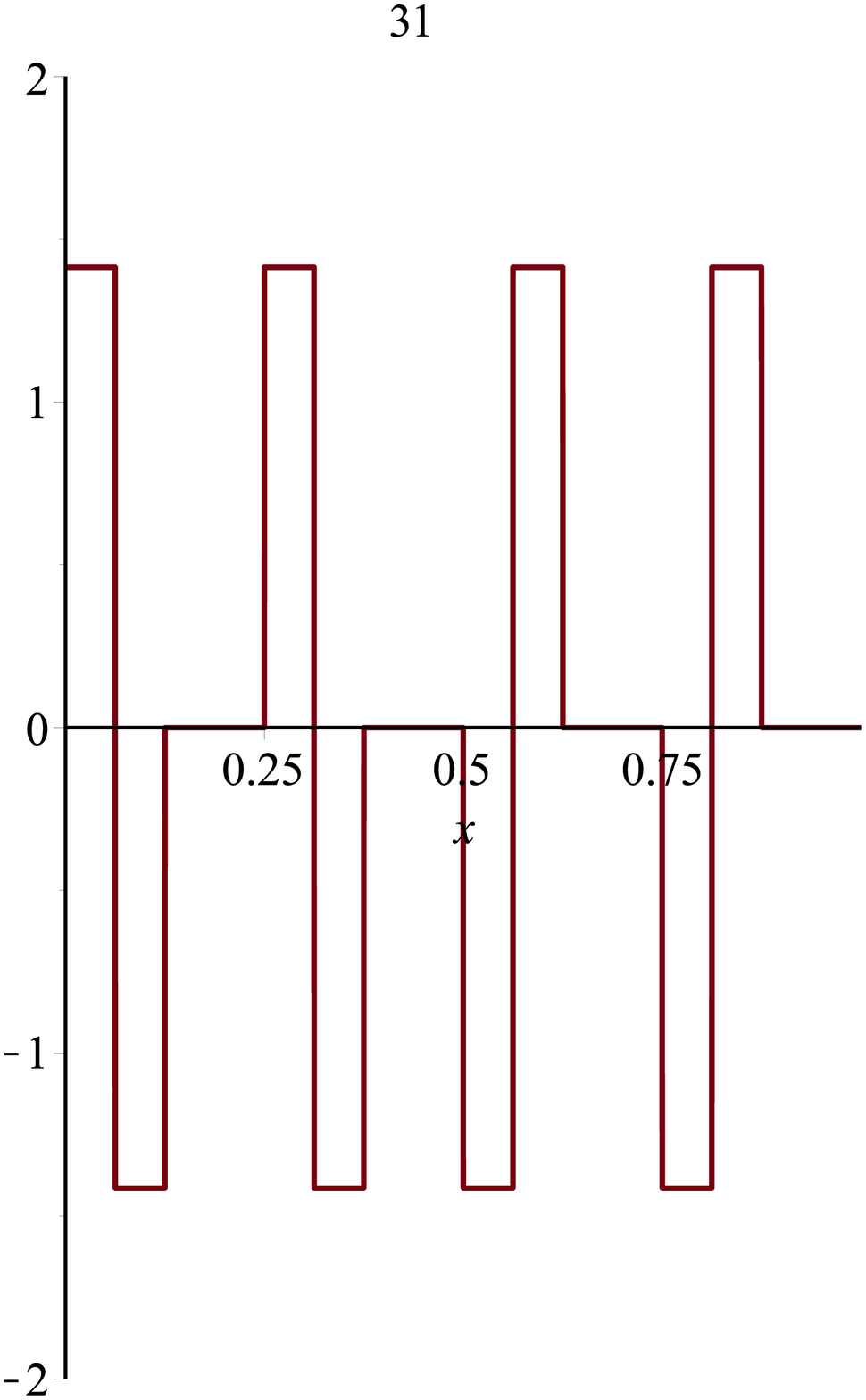}
    \includegraphics[width=1.2in]{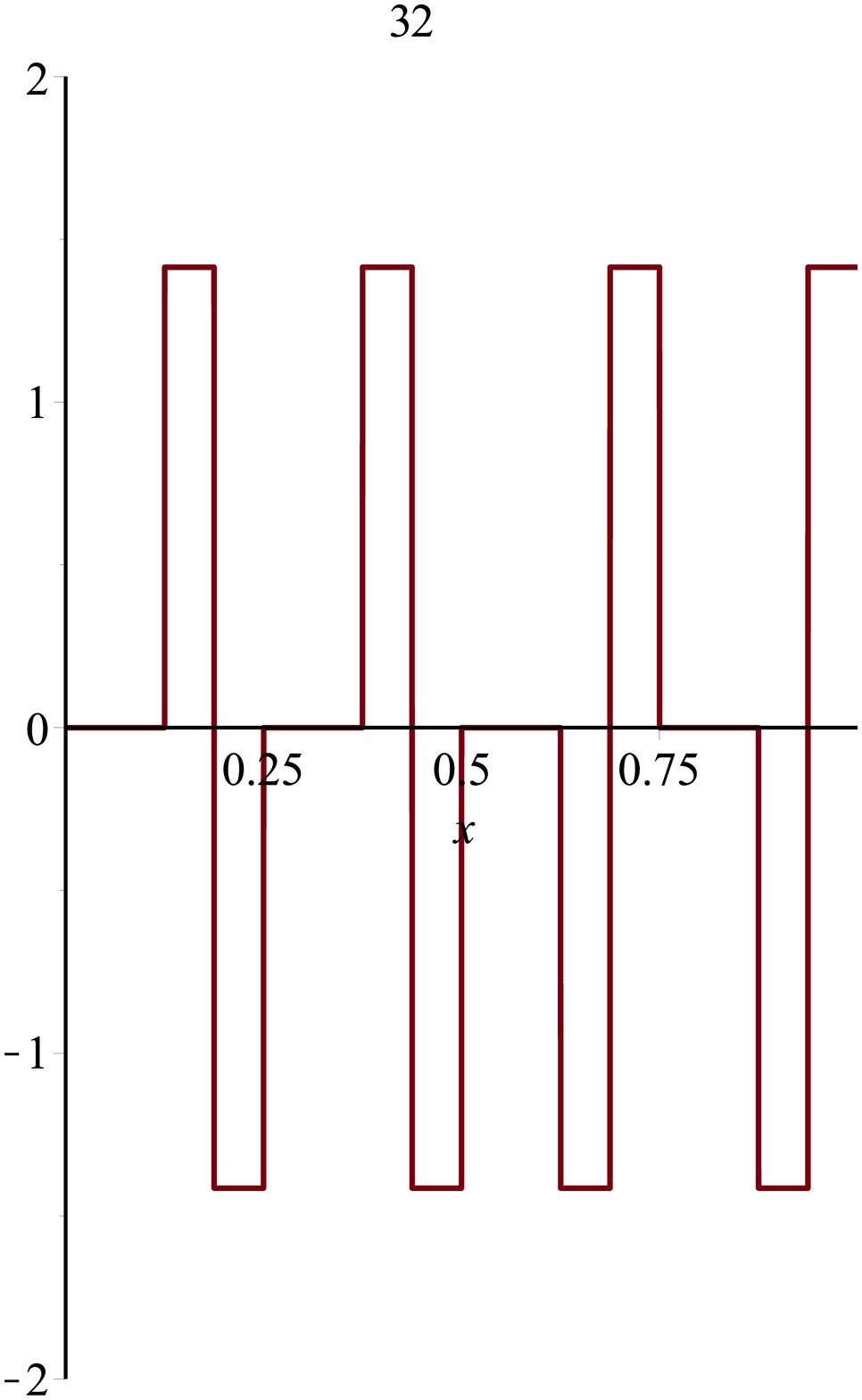}
    \includegraphics[width=1.2in]{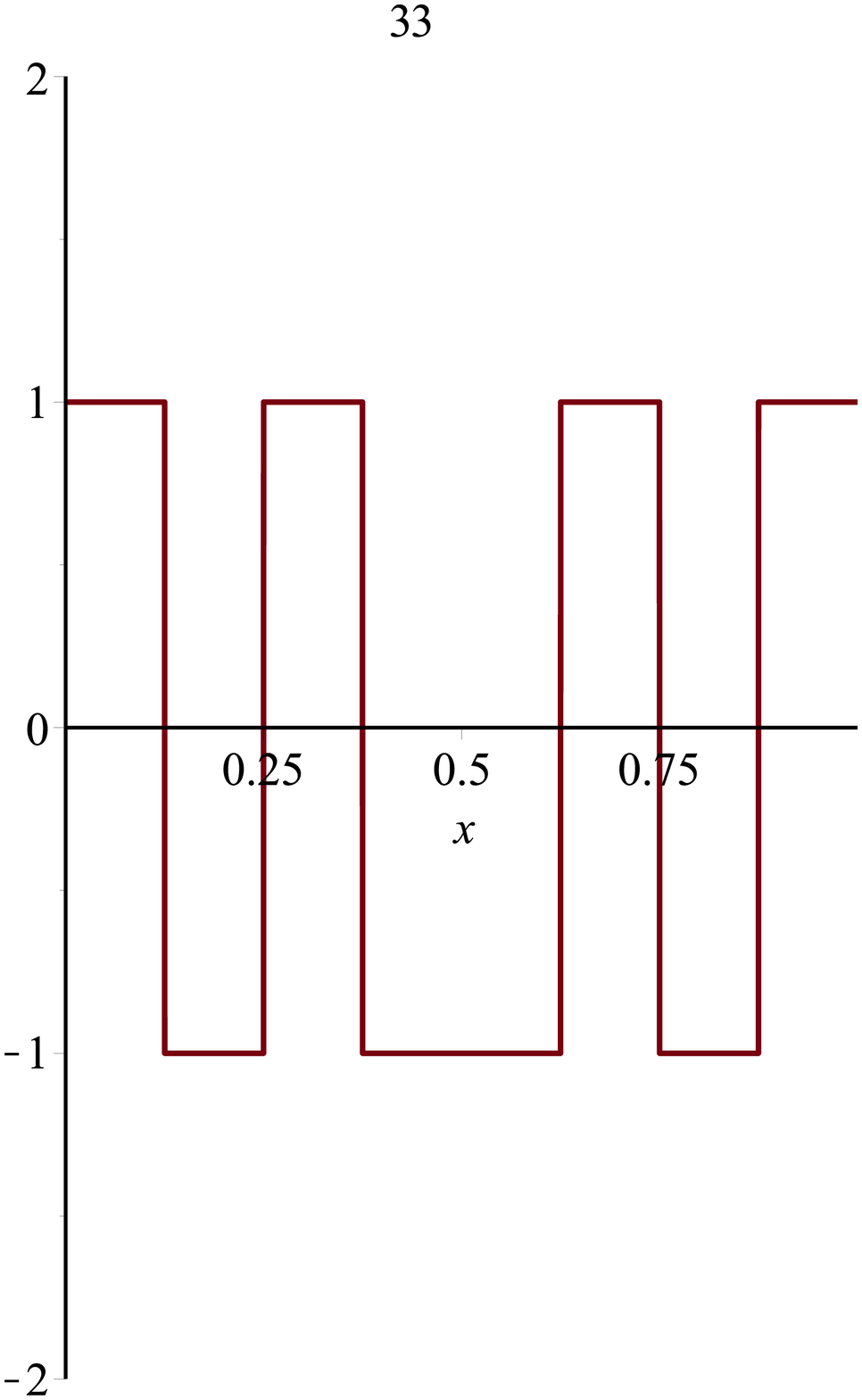}
    \newline
  \end{center}
  \caption{Walsh functions $S_w1$ for words $w$ of length 2.}
\end{figure}
%\clearpage

\begin{example}
The pictures in Figure 1 show the Walsh functions that correspond to the scale $N=4$ and the matrix
$$A=\begin{pmatrix}
	\frac{1}{2}&\frac12&\frac12&\frac12\\
	\frac{\sqrt2}2&-\frac{\sqrt2}2&0&0\\
	0&0&\frac{\sqrt2}2&-\frac{\sqrt2}2\\
	\frac12&\frac12&-\frac12&-\frac12
\end{pmatrix}
$$
for the words of length 2, indicated at the top. 

\end{example}

\bibliographystyle{alpha}	
\bibliography{eframes}
\end{document}